\documentclass[reqno,11pt,psamsfonts]{amsart}

\usepackage[all]{xy}
\usepackage{amsthm}
\usepackage{amssymb}
\usepackage{amsmath,amscd}
\usepackage{geometry}

\def\cal{\mathcal}

\def\AA{{\Bbb A}}

\def\NN{{\Bbb N}}

\def\ZZ{{\Bbb Z}}

\def\11{{1\kern-3.5pt 1}}
\def\mumu{{\mu\kern-4.2pt\mu}}
\def\boxtimes{\setbox0\hbox{$\Box$}\copy0\kern-\wd0\hbox{$\times$}}

\newcommand{\lotimes}{\otimes^{\bf{L}}}

\def\diag{\operatorname {diag}}

\def\End{\operatorname {E nd}}
\def\Ext{\operatorname {Ext}}

\def\GL{\operatorname {GL}}

\def\Hom{\operatorname {Hom}}

\def\Ker{\operatorname {ker}}

\def\Spec{\operatorname {Spec}}

\def\Aut{\operatorname{Aut}}

\def\fchar{\operatorname{char}}

\def\depth{\operatorname{depth}}

\def\dim{\operatorname{dim}}

\def\End{\operatorname{End}}

\def\Ext{\operatorname{Ext}}

\def\uExt{\operatorname{\underline{Ext}}}

\def\gcd{\operatorname{gcd}}

\def\GL{\operatorname{GL}}
\def\gldim{\operatorname{gldim}}
\def\grmod{\operatorname{grmod}}

\def\GrAut{\operatorname{GrAut}}
\def\GrMod{\operatorname{GrMod}}

\def\Hom{\operatorname{Hom}}

\def\uHom{\operatorname{\underline{Hom}}}

\def\Im{\operatorname{Im}}

\def\injdim{\operatorname{injdim}}

\def\Ker{\operatorname{Ker}}
\def\lcm{\operatorname{lcm}}

\def\max{\operatorname{max}}

\def\mod{\operatorname{mod}}

\def\Proj{\operatorname{Proj}}

\def\sup{\operatorname{sup}}

\def\tails{\operatorname{tails}}
\def\Tails{\operatorname{Tails}}

\def\tors{\operatorname{tors}}
\def\Tors{\operatorname{Tors}}

\def\uEnd{\operatorname{\underline{End}}}
\def\uExt{\operatorname{\underline{Ext}}}

\def\uH{\operatorname{\underline{H}}}
\def\uHom{\operatorname{\underline{Hom}}}
\def\RuHom{\operatorname{R\underline{Hom}}}

\def\uG{\operatorname{\underline{\Gamma}}}

\def\R{\operatorname{R}}

\let\oldtext\text
\def\text#1{\oldtext{\normalshape #1}}

\def\a{\alpha}
\def\b{\beta}

\def\e{\epsilon}

\def\s{\sigma}
\def\t{\tau}

\def\L{\Lambda}

\def\fm{{\frak m}}

\def\cA{{\cal A}}
\def\cB{{\cal B}}
\def\cC{{\cal C}}
\def\cD{{\cal D}}

\def\cF{{\cal F}}

\def\cK{{\cal K}}
\def\cL{{\cal L}}
\def\cM{{\cal M}}
\def\cN{{\cal N}}
\def\cO{{\cal O}}

\def\cS{{\cal S}}

\def\SL{\operatorname{SL}}

\def\<{\langle}
\def\>{\rangle}
\def\CM{\operatorname {CM}}
\def\uCM{\underline {\operatorname
{CM}}}
\def\HSL{\operatorname{HSL}}
\def\hdet{\operatorname{hdet}}
\def\rad{\operatorname{rad}}

\newtheorem{lemma}{Lemma}[section]

\newtheorem{theorem}[lemma]{Theorem}
\newtheorem{corollary}[lemma]{Corollary}

{

}

\theoremstyle{definition}

\newtheorem{example}[lemma]{Example}
\newtheorem{definition}[lemma]{\sl Definition}

\theoremstyle{remark}

\newtheorem{remark}[lemma]{Remark}

\begin{document}
\pagenumbering{arabic}

\title[Ample Group Action on AS-regular Algebras]{Ample Group Action on AS-regular Algebras and Noncommutative Graded Isolated Singularities}

\author{Izuru Mori}

\address{Department of Mathematics, Faculty of Science, Shizuoka University, 836 Ohya, Suruga-ku, Shizuoka 422-8529, JAPAN}

\email{simouri@ipc.shizuoka.ac.jp}

\author{Kenta Ueyama}

\address{
Department of Mathematics,
Graduate School of Science,
Shizuoka University,
836 Ohya, Suruga-ku,
Shizuoka 422-8529, JAPAN}
\email{skueyam@ipc.shizuoka.ac.jp} 

\address{Current address of the second author:
Department of Mathematics,
Faculty of Education,
Hirosaki University,
1 Bunkyocho, Hirosaki,
Aomori 036-8560, JAPAN}
\email{k-ueyama@cc.hirosaki-u.ac.jp}



\thanks {The first author was supported by Grant-in-Aid for Scientific
Research (C) 25400037.
The second author was supported by JSPS Fellowships for Young Scientists
No.\ 23-2233.}



\begin{abstract} In this paper, we introduce a notion of ampleness of a group action $G$ on a right noetherian graded algebra $A$, and show that it is strongly related to the notion of $A^G$ to be a graded isolated singularity introduced by the second author of this paper.  Moreover, if $S$ is a noetherian AS-regular algebra and $G$ is a finite ample group acting on $S$,
then we will show that ${\mathcal D}^b(\operatorname{tails} S^G)\cong {\cal D}^b(\operatorname{mod} \nabla S*G)$ where $\nabla S$ is the Beilinson algebra of $S$.
We will also explicitly calculate a quiver $Q_{S, G}$ such that ${\mathcal D}^b(\operatorname{tails} S^G)\cong {\mathcal D}^b(\operatorname{mod} kQ_{S, G})$ when $S$ is of dimension 2. 
\end{abstract}

\maketitle


\section{Introduction}


In commutative ring theory, regular rings, Gorenstein rings, and Cohen-Macaulay rings are fundamental classes of rings that are studied.  Since these classes of rings can be characterized by homological algebra, noncommutative (graded) generalizations of these rings, namely, AS-regular algebras, AS-Gorenstein algebras, and AS-Cohen-Macaulay algebras, have been defined and intensively studied in noncommutative algebraic geometry.  On the other hand, although isolated singularities are also an important class of rings in commutative ring theory, its notion is geometric in nature, so they have not been generalized nor studied well in noncommutative algebraic geometry.  Recently, the second author of this paper introduced a notion of graded isolated singularity for (noncommutative) graded algebras, which agrees with the usual notion of isolated singularity if the algebra is commutative and generated in degree 1, and found some nice properties of such algebras in \cite {U}.  This paper continues to study (noncommutative) graded isolated singularities. 

In commutative ring theory and representation theory of algebras, the fixed subalgebra $S^G$ of a polynomial algebra $S=k[x_1, \dots, x_d]$ by a finite subgroup $G\leq \GL(d, k)$ has been intensively studied.  In particular, it is known that $S^G$ is Gorenstein if and only if $G\leq \SL(d, k)$, and $S^G$ is an isolated singularity if and only if $G$ acts freely on $\AA^d\setminus \{0\}$.   Since an AS-regular algebra $S$ is a noncommutative generalization of a polynomial algebra, it is natural to study a fixed subalgebra $S^G$ of an AS-regular algebra $S$ by a finite group $G$ acting on $S$.   In \cite {JZ}, Jorgensen and Zhang introduced the notion of homological determinant for a graded algebra automorphism of $S$, which is the same as the usual determinant if $S$ is a polynomial algebra generated in degree 1, and showed that, for a noetherian AS-regular algebra $S$, if every element of $G$ has homological determinant 1, then $S^G$ is AS-Gorenstein.  In this paper, we will introduce a notion of ampleness of $G$ for $S$, and show 
that, for a noetherian AS-regular algebra $S$ of dimension $d\geq 2$ and a finite group $G$ consisting of elements of homological determinant 1, $G$ is ample for $S$ if and only if $S^G$ is a graded isolated singularity and $S*G\cong \End_{S^G}(S)$ (Theorem \ref{thm.ampeq}).  
If $S$ is a polynomial algebra generated in degree 1, then $S * G\cong \End_{S^G}(S)$
in the above setting by \cite [Theorem 3.2]{IT}, so $G$ is ample for $S$ if and only if $S^G$ is an isolated singularity in the usual sense (Corollary \ref{cor.ampeq}).  

Let $S=k[x_1, \dots, x_d]$ be a (weighted) polynomial algebra and $G\leq \GL(d, k)$ a finite subgroup acting on $S$.  The derived category $\cD^b(\tails S^G)$ and the graded stable category $\uCM ^{\ZZ}(S^G)$ have been studied in two situations, namely, when $S$ is generated in degree 1 in \cite {IT}, and when $G=\<\diag (\xi^{\deg x_1}, \dots, \xi^{\deg x_d})\>$ is a finite cyclic group where $\xi\in k^\times$ is a primitive $r$-th root of unity so that $S^G=S^{(r)}$ is the $r$-th Veronese algebra of $S$ in \cite {U1} and \cite [Section 5]{AIR} (see \cite {Mmc} when $S$ is AS-regular).  One of the motivations of this paper is to integrate both situations to find a common theory.  As one of such common theories, we show in this paper that $\cD^b(\Proj S^G):=\cD^b(\tails S^G)\cong \cD^b(\mod \nabla S*G)$ when $S$ is a noetherian AS-regular algebra of dimension $d\geq 2$ and $G$ is a finite ample group acting on $S$ where $\nabla S$ is the Beilinson algebra of $S$ (Theorem \ref{thm.tilt}).  As a main ingredient of the proof, we also show that if $S$ is a noetherian AS-regular algebra, then $S*G$ is a noetherian AS-regular algebra over $kG$ as introduced in \cite {MM} (Corollary \ref{cor.AS}).  
If $G$ is ample, then $S^G$ is a graded isolated singularity, so the noncommutative projective scheme $\Proj S^G$ associated to $S^G$ is smooth by definition.  It follows that $\Proj S^G$ itself is the minimal resolution of $\Proj S^G$, so the above derived equivalence can be thought of as a noncommutative projective version of (algebraic) McKay correspondence (see \cite{Mmc} for details).    

In the last section of this paper, we study the case when $S$ is a noetherian AS-regular algebra of dimension 2.  If $S=k[x, y]$ is the polynomial algebra and $G\leq \SL(2, k)$, then it is known that $S^G$ is an isolated singularity (so $G$ is ample for $S$ if $\deg x=\deg y=1$).  In this paper, we show that if $S\not \cong k\<x, y\>/(xy\pm yx)$, then every finite group $G$ consisting of elements of homological determinant 1 is ample for $S$ (Theorem \ref{thm.4.5}), and, in this case, we explicitly calculate a quiver $Q_{S, G}$ such that $\cD^b(\tails S^G)\cong \cD^b(\mod kQ_{S, G})$ (Theorem \ref{lem.QSG}).  We will continue to study  $\uCM ^{\ZZ}(S^G)$ in the subsequent paper.  

\smallskip \noindent {\it Acknowledgments:} We would like to thank Hideto Asashiba and Osamu Iyama for helpful discussions during the preparation of this paper.  

\subsection{Terminologies and Notations}

Throughout this paper, we fix a field $k$.
Unless otherwise stated, a graded algebra means an $\NN$-graded algebra $A =\bigoplus_{i\in\NN} A_i$ over $k$.
The group of graded $k$-algebra automorphisms of $A$ is denoted by
$\GrAut A$. 
We denote by $\GrMod A$ the category
of graded right A-modules, and by $\grmod A$ the full subcategory consisting of
finitely generated modules. Morphisms in $\GrMod A$ are right $A$-module homomorphisms of degree zero.
Graded left $A$-modules are identified with graded $A^o$-modules where $A^o$ is the opposite graded algebra of $A$.  For
$M \in \GrMod A$ and $n \in \ZZ$, we define $M_{\geq n}$ = $\bigoplus_{i\geq n} M_i \in \GrMod A$, and $M(n) \in
\GrMod A$ by $M(n) = M$ as an ungraded right $A$-module with the new grading
$M(n)_i = M_{n+i}$.
The rule $M \mapsto  M(n)$ is a
$k$-linear autoequivalence for $\GrMod A$ and $\grmod A$, called the shift functor.
For $M,N \in \GrMod A$, we write the vector space
$\Ext^i_A(M,N) = \Ext^i_{\GrMod A}(M,N)$
and the graded vector space
$$\uExt^i_A(M,N) :=\bigoplus_{n\in\ZZ}\Ext^i_A(M,N(n)).$$
For a graded right (resp. left) $A$-module $M$,
we denote by $DM:=\uHom _k(M, k)$ the dual graded vector space
of $M$. Note that $DM$ has a graded
left (resp. right) $A$-module structure.
Let $\s\in \GrAut A$ be a graded algebra
automorphism. For a graded right $A$-module $M\in \GrMod A$, we
define a new graded right $A$-module $M_{\s}\in \GrMod A$ by
$M_{\s}=M$ as a graded vector space with the new right action
$m*a:=m\s(a)$ for $m\in M$ and $a\in A$.  If $M$ is a graded $A$-$A$
bimodule, then $M_{\s}$ is also a graded $A$-$A$ bimodule by this
new right action.

We say that $M \in \GrMod A$ is torsion if, for any $m \in M$, there exists $n \in \NN$
such that $mA_{\geq n} = 0$. We denote by $\Tors A$ the full subcategory of $\GrMod A$
consisting of torsion modules.
We define the torsion functor $\uG_{\fm} : \GrMod A \to \Tors A$ by
$$\uG_{\fm}(M)= \lim _{n \to \infty}\uHom_A(A/A_{\geq n}, M).$$
We define the local cohomology modules of $M \in \GrMod A$ by
$$
\uH^i_\fm(M)= \R^i\uG_\fm(M)
= \lim _{n \to \infty} \uExt^i_A(A/A_{\geq n}, M),
$$
and the depth of $M$ by 
$$\depth M=\inf\R\uG_{\fm}(M)=\inf\{i\in \NN\mid \uH_{\fm}^i(M)\neq 0\}.$$ 

We write $\Tails A$ for the quotient category $\GrMod(A) / \Tors(A)$.
The quotient functor is denoted by $\pi :\GrMod A \to \Tails A$. 
The objects in $\Tails A$ will be denoted by script letters, like $\cM = \pi M$.
The shift functor on $\GrMod A$ induces an autoequivalence $(n) :\cM \mapsto \cM(n)$ on
$\Tails A$ which we also call the shift functor.
For $\cM,\cN \in \ \Tails A$, we write the vector space
$\Ext^i_{\cA}(\cM,\cN) = \Ext^i_{\Tails A}(\cM,\cN)$
and the graded vector space
$$\uExt^i_{\cA}(\cM,\cN) :=\bigoplus_{n\in\ZZ}\Ext^i_{\cA}(\cM,\cN(n)).$$
When $A$ is right noetherian, 
we define $\tails A := \grmod A/\tors A$ where $\tors A = \Tors A \cap \grmod A$. 
We define the global dimension of $\tails A$ by
$$ \gldim (\tails A) = \sup\{i \mid \Ext^i_{\cA}(\cM, \cN) \neq0\; \textnormal{for some}\;\cM, \cN \in \tails A \}.$$
It is easy to see that if $A$ has finite global dimension,
then $\tails A$ has finite global dimension.

\begin{definition} (\cite [Definition 2.2]{U})
A right noetherian graded algebra $A$ is called a graded isolated singularity
if $\gldim (\tails A)< \infty$.
\end{definition}

If $A$ is a commutative graded algebra finitely generated in degree 1, then $A$ is graded isolated singularity if and only if $A$ is an isolated singularity in the usual sense.

If $A_0 = k$, then we say that $A$ is connected graded.
Let $A$ be a noetherian connected graded algebra.
Then we view $k = A/ A_{\geq 1} \in \GrMod A$
as a graded A-module.

\begin{definition} \label{def.as} 
A 
connected graded algebra $A$ is called an AS-Gorenstein (resp. AS-regular) algebra of dimension $d$ and of Gorenstein parameter $\ell$ if
\begin{itemize}
\item{} $\injdim_A A = \injdim_{A^o} A= d <\infty$ (resp. $\gldim A=d<\infty$), and
\item{} $\uExt^i_A(k ,A) \cong \uExt^i_{A^o}(k ,A) \cong
\begin{cases}
k(\ell) & \textnormal { if }\; i=d, \\
0 & \textnormal { if }\; i\neq d.
\end{cases}$
\end{itemize}
\end{definition}



If $A$ is a noetherian AS-Gorenstein algebra of dimension $d$ and of Gorenstein
parameter $\ell$, then $\uH_{\fm}^i(A)=0$ for all $i\neq d$,
and $\uH^d_\fm(A) \cong DA(\ell)$ as a graded right and left module.
We say that $M \in \grmod A$ is graded maximal Cohen-Macaulay if 
$\uH_{\fm}^i(M)=0$ for all $i\neq d$.
We denote by $\CM^{\ZZ} (A)$ the full subcategory of $\grmod A$
consisting of graded maximal Cohen-Macaulay modules.

Let $A$ be a graded algebra, $\s \in \GrAut A$, and $M, N \in \GrMod A$.
A $k$-linear graded map $f: M\rightarrow N$ is called $\s$-linear if
$f: M\rightarrow N_{\s}$
is a graded $A$-module homomorphism.
If $A$ is AS-Gorenstein,
then by \cite[Lemma 2.2]{JZ}, $\s: A \to A$ induces a $\s$-linear map
$\uH^i_\fm(\s): \uH_{\fm}^i(A) \rightarrow \uH^i_\fm(A)$.
Moreover, there exists a constant $c \in k^{\times}$ such that 
$\uH^d_\fm(\s): \uH^d_\fm(A) \rightarrow \uH^d_\fm(A)$
is equal to 
$ cD(\s^{-1}): DA(\ell) \to DA(\ell).$
The constant $c^{-1}$ is called the homological determinant of $\s$,
and we denote $\hdet \s = c^{-1}$ (see \cite[Definition 2.3]{JZ}).
By \cite[Lemma 2.5]{JZ}, $\hdet$ defines a group homomorphism $\GrAut A \to k^{\times}$.
We define the homological special linear group on $A$ by the kernel of $\hdet$
$$\HSL(A)= \{ \s \in \GrAut A \mid \hdet \s =1\}.$$

If $S$ is a commutative AS-regular algebra of dimension $d$, then $S=k[x_1, \dots, x_d]$ is a polynomial algebra.  If it is generated in degree 1, then 
$\GrAut S=\GL(d, k)$ and $\hdet $ coincides with the usual determinant of a matrix so that $\HSL(S)=\SL(d, k)$. 

\begin{theorem} \textnormal{(\cite[Theorem 3.3]{JZ})} \label{thm.asg}
If $A$ is a noetherian AS-Gorenstein algebra of dimension $d$ and of Gorenstein parameter $\ell$, and $G\leq \HSL(A)$ is a finite subgroup, 
then the fixed subalgebra $A^G$ is a noetherian AS-Gorenstein algebra of dimension $d$ and of Gorenstein parameter $\ell$.
\end{theorem}

\subsection{Conventions on Group Actions}

In this subsection, we give some conventions on group actions on algebras used in this paper. 
Let $A$ be a graded algebra and $G\leq \GrAut A$ a
finite subgroup.   We will always assume that $\fchar k$ does not divide $|G|$ so that $e:=\frac{1}{|G|}\sum _{g\in
G}1*g\in A*G$ is a well-defined idempotent.  Note that this
condition is equivalent to the condition that $kG$ is semi-simple.  
Note also that $A^G$ and $A*G$ are graded by 
$(A^G)_i=A^G\cap A_i$ and $(A*G)_i=A_i\otimes _kkG$ for $i\in \NN$.  

The following lemma is well-known.  These identifications of algebras and modules play essential roles in this paper.

\begin{lemma}  \label{lem.eGe}\; \label{lem.right}
Let $A$ be a graded algebra and $G\leq \GrAut A$ a finite subgroup.  Suppose that
$\fchar k$ does not divide $|G|$ so that $e:=\frac{1}{|G|}\sum _{g\in G}1*g\in
A*G$ is well-defined.  
\begin{enumerate}
\item{} The map $\varphi: A^G\to e(A*G)e$ defined by $\varphi (c)=e(c*1)e$
is an isomorphism of graded algebras. 
\item{} The map $\psi: A\to (A*G)e$
defined by $\psi (a)=(a*1)e$ is an isomorphism of graded right $A^G$-modules where the right $A^G$-module structure on $(A*G)e$ is given by identifying $A^G$ with $e(A*G)e$ via $\varphi$. 
\item{} 
The map $\phi :A\to e(A*G)$ defined by $\phi(a)=e(a*1)$ is an isomorphism of graded left $A^G$-modules where the left $A^G$-module structure on $e(A*G)$ is given by identifying $A^G$ with $e(A*G)e$ via $\varphi$. 
\end{enumerate}
\end{lemma} 

We also endow a right $A*G$-module structure on
$A$ by identifying $A$ with $e(A*G)$ via the above $\phi$ so that, for $a\in A, b*g\in A*G$, $a\cdot (b*g)=g^{-1}(ab)$.  
By Lemma \ref{lem.eGe}, 
$$A^G\cong
e(A*G)e\cong \uEnd _{A*G}(e(A*G))\cong \uEnd _{A*G}(A)$$ as algebras.  We will see in this paper that
$$A*G\cong \uEnd _{A^G}((A*G)e)\cong \uEnd_{A^G}(A)$$
as algebras in some nice situations (algebraic McKay correspondence).

\section{Ampleness} 

We first review the notion of ampleness introduced in \cite {AZ}.  An algebraic triple $(\cC, \cO, s)$ consists of a  
$k$-linear abelian category $\cC$, an object $\cO\in \cC$, and a
$k$-linear autoequivalence $s\in \Aut _k\cC$. Given a right noetherian graded algebra $A$, the noncommutative projective scheme associated to $A$ is defined by the algebraic triple $\Proj A=(\tails A, \cA, (1))$.  It is important to find a better homogeneous coordinate ring of $\Proj A$, that is, to find a better graded algebra $B$ (eg. $\gldim B<\infty$) such that 
$\tails B\cong \tails A$.  For a suitable choice of an ample pair $(\cO, s)$ for $\tails A$ defined below, the graded algebra $B=B(\tails A, \cO, s)$ constructed in Definition \ref{def.B} may be a better homogeneous coordinate ring than $A$. 

\begin{definition} (\cite{AZ})
Let $(\cC, \cO, s)$ be an algebraic triple.  We say that the pair
$(\cO, s)$ is ample for $\cC$
if \\

\noindent (A1) for every object $\cM\in \cC$, there are positive integers
$r_1, \dots , r_p\in \NN^+$ and an epimorphism $\bigoplus
_{i=1}^ps^{-r_i}\cO\to \cM$ in $\cC$, and \\

\noindent (A2) for every epimorphism $\cM\to \cN$ in $\cC$, there is an
integer $n_0$ such that the induced map $\Hom _{\cC}(s^{-n}\cO,
\cM)\to \Hom _{\cC}(s^{-n}\cO, \cN)$ is surjective for every $n\geq
n_0$.
\end{definition}

\begin{definition} \label{def.B}
We define the graded
algebra associated to an algebraic triple $(\cC, \cO, s)$ by
$$B(\cC, \cO, s):=\bigoplus _{i\in \ZZ}\Hom_{\cC}(\cO, s^i\cO).$$
\end{definition}

This gives a functor from the category of algebraic triples to the
category of graded algebras.  
Note that $\bigoplus _{i\in \ZZ}\Hom_{\cC}(\cO, s^i\cF)$ has a natural graded right $B(\cC, \cO,
s)$-module structure for any object $\cF\in \cC$.  

The $\chi$-condition below plays an essential role in noncommutative algebraic geometry.  

\begin{definition} (\cite {AZ})
We say that a right noetherian graded algebra $A$ satisfies $\chi_i$
if $\uExt^j_A(T, N)$ are finite dimensional for all $T\in \tors A$,
$N\in \grmod A$ and $0\leq j\leq i$.
\end{definition}

It is known that every noetherian AS-Gorenstein algebra satisfies $\chi_i$ for all $i\in \NN$.  The theorem below justifies the notion of ampleness.  

\begin{theorem} \textnormal{(\cite[Corollary 4.6 (1)]{AZ})} \label{thm.AZ}
Let $(\cC, \cO, s)$ be an
algebraic triple.  If $\cO\in \cC$ is a noetherian object, $\dim _k\Hom_{\cC}(\cO, \cM)<\infty$ for all $\cM\in \cC$, and $(\cO, s)$ is ample for $\cC$, then
$B:=B(\cC, \cO, s)_{\geq 0}$ is a locally finite right noetherian graded algebra satisfying $\chi_1$, and 
the functor
$$\cC\to \tails B; \quad \cF\mapsto \pi \left(\bigoplus _{i\in
\NN}\Hom_{\cC}(\cO, s^i\cF) \right)$$ induces an equivalence of algebraic
triples $(\cC, \cO, s)\to (\tails B, \cB, (1))=\Proj B$.
\end{theorem}

\subsection{Tails of a Graded Endomorphism Algebra} 
In this subsection, we will consider a Morita type question, that is, finding a condition on $M\in \grmod A$ such that $\tails \uEnd_A(M)\cong \tails A$.  
Since $B(\grmod A, M, (1))=\uEnd _A(M)$ as graded
algebras, the answer is given by the ampleness of $(\cM, (1))$ for $\tails A$.  

Let $A, B$ be graded algebras and $M$ a graded $A$-$B$ bimodule.  We write 
\begin{align*}
-\otimes _AM:\Tails A\to \Tails B  \\
\uHom_B(M, -):\Tails B\to \Tails A
\end{align*}
for the functors induced by the functors
\begin{align*}
-\otimes _AM:\GrMod A\to \GrMod B  \\
\uHom_B(M, -):\GrMod B\to \GrMod A
\end{align*}
by abuse of notations. 

\begin{theorem} \label{thm.aend}  
Let $A$ be a noetherian AS-Gorenstein graded isolated singularity of dimension $d\geq 2$, and $M\in \CM^{\ZZ}(A)$.  If $M$ contains $A$ as a direct summand, then $\uHom_A(M, -):\tails A\to \tails \uEnd_A(M)$ is an equivalence functor. 
\end{theorem}  

\begin{proof} 
Note that $\cA\in \tails A$ is a noetherian object.  Since $A$ is a locally finite right noetherian graded algebra satisfying $\chi_1$, we have that $\dim _k\Hom_{\cA}(\cA, \cM)<\infty$ for all $\cM\in \tails A$ by \cite [Corollary 4.6 (2)]{AZ}.  We will now show that $(\cM, (1))$ is ample for $\tails A$.  Since $\cM$ contains $\cA$ as a direct summand, it is easy to see that the condition (A1) is satisfied.  Note that every exact sequence $0\to \cK\to \cL\to \cN\to 0$ in $\tails A$ is induced by an exact sequence $0\to K\to L\to N\to 0$ in $\grmod A$.   Consider the exact sequence 
$$\uHom_A(M, L)\to \uHom_A(M, N)\to \uExt^1_A(M, K).$$  Since $A$ is a noetherian AS-Gorenstein graded isolated singularity and $M\in \CM^{\ZZ}(A)$, it follows that $\uExt^1_A(M, K)$ is finite dimensional over $k$ by \cite [Lemma 5.7]{U}, so $\uExt^1_{\cA}(\cM, \cK)$ is right bounded by \cite [Corollary 7.3 (2)]{AZ}.
It follows that 
$$\Ext^1_{\cA}(\cM(-n), \cK)=\uExt^1_{\cA}(\cM, \cK)_n=0$$
for all $n\gg 0$.  Since 
$$\Hom_{\cA}(\cM(-n), \cL)\to \Hom_{\cA}(\cM(-n), \cN)\to \Ext^1_{\cA}(\cM(-n), \cK)$$ is exact, $\Hom_{\cA}(\cM(-n), \cL)\to \Hom_{\cA}(\cM(-n), \cN)$ is surjective for all $n\gg 0$, so the condition (A2) is also satisfied, hence $(\cM, (1))$ is ample for $\tails A$.

Since $\depth _AM=d\geq 2$, 
$$\uEnd_A(M)=B(\grmod A, M, (1))\cong B(\tails A, \cM, (1))$$
by \cite [Lemma 3.3]{Mmc}.  By Theorem \ref{thm.AZ}, $\uEnd _A(M)_{\geq 0}\cong  B(\tails A, \cM, (1))_{\geq 0}$
is locally finite right noetherian and the functor 
$$\pi \circ \uHom_{\cA}(\cM, -):\tails A\to \tails \uEnd_A(M)_{\geq 0}$$ is an equivalence functor.  Since $\uEnd_A(M)$ is locally finite and left bounded by \cite [Proposition 3.1]{AZ}, $\uEnd _A(M)/\uEnd_A(M)_{\geq 0}$ is finite dimensional over $k$, so $\uEnd _A(M)$ is finitely generated as a graded right $\uEnd_A(M)_{\geq 0}$-module.  It follows that $\uEnd _A(M)$ is right noetherian, so the natural functor $\tails \uEnd _A(M)\to  \tails \uEnd_A(M)_{\geq 0}$ is an equivalence functor by \cite [Proposition 2.5 (2)]{AZ}, hence the functor 
$$\pi \circ \uHom_{\cA}(\cM, -):\tails A\to \tails \uEnd_A(M)$$ is an equivalence functor.  For $N\in \grmod A$, there exists $n\in \ZZ$ such that
$$\uHom_{\cA}(\cM, \cN)_{\geq n}\cong \uHom_A(M, N)_{\geq n}$$
in $\grmod \uEnd_A(M)$ by \cite [Corollary 7.3 (2)]{AZ}, so the above equivalence functor is induced by the functor $\uHom_A(M, -):\grmod A\to \grmod \uEnd_A(M)$.  
\end{proof} 

\begin{corollary} Let $S$ be a noetherian AS-regular algebra of dimension $d\geq 2$, and $G\leq \HSL(S)$.  If $S^G$ is a graded isolated singularity, then $\uHom_{S^G}(S, -):\tails S^G\to \tails \uEnd_{S^G}(S)$ is an equivalence functor. 
\end{corollary} 

\begin{proof} Since $G\leq \HSL(S)$, $S^G$ is a noetherian AS-Gorenstein graded isolated singularity of dimension $d\geq 2$ by Theorem \ref{thm.asg}.  Since $S\in \CM^{\ZZ}(S^G)$ contains $S^G$ as a direct summand, the result follows from Theorem \ref{thm.aend}. 
\end{proof} 

We will consider another situation.  Recall that if $e\in A$ is an idempotent, then $\uEnd_A(eA)\cong eAe$ as graded algebras, so we will now find condition on $e\in A$ such that $\tails eAe\cong \tails A$. 

The following lemma is known to the experts.  The proof is similar to that of Lemma \ref{lem.ample2}.  

\begin{lemma} \label{lem.ample}
Let $A$ be a right noetherian graded algebra and $e\in A$ an idempotent such that $eAe$ is a right noetherian graded algebra.
If $eA$ is finitely generated as a graded left $eAe$-module and $A/(e)\in \Tors A$, then
$$-\otimes _AAe: \Tails A\rightleftarrows
\Tails eAe: -\otimes _{eAe}eA$$ are equivalence functors
quasi-inverse to each other.  Moreover, if $Ae\in \grmod eAe$, then
$$-\otimes _AAe: \tails
A\rightleftarrows \tails eAe: -\otimes _{eAe}eA$$ are equivalence functors
quasi-inverse to each other.
\end{lemma}

For the purpose of this paper, we modify the $\chi$-condition as
follows.

\begin{definition}
We say that a graded algebra $A$ satisfies $\widetilde {\chi}_i$ if
$\uExt^j_A(T, N)\in \Tors A$ for all $T\in \GrMod A^e$ such that
$T_A\in \Tors A$, $N\in \GrMod A$ and $0\leq j\leq i$.
\end{definition}

\begin{lemma} \label{lem.ample2}
Let $A$ be a right noetherian graded algebra and $e\in A$ an idempotent such that $eAe$ is a right noetherian graded algebra.
If $A$ satisfies $\widetilde {\chi}_1$, $eA$ is finitely generated as a graded left $eAe$-module and $A/(e)\in \Tors A$,
then
$$\uHom_A(eA, -): \Tails A\rightleftarrows \Tails eAe:
\uHom_{eAe}(Ae, -)$$ are equivalence functors quasi-inverse to each
other.  Moreover, if $Ae\in \grmod eAe$,  then
$$\uHom_A(eA, -): \tails A\rightleftarrows \tails eAe:
\uHom_{eAe}(Ae, -)$$ are equivalence functors quasi-inverse to each
other.
\end{lemma}

\begin{proof}
For $M\in \GrMod eAe$,
$$\uHom_A(eA, \uHom_{eAe}(Ae, M))\cong \uHom_{eAe}(eA\otimes _AAe,
M)\cong \uHom_{eAe}(eAe, M)\cong M$$ in $\GrMod eAe$.

Let $\psi :Ae\otimes_{eAe}eA\to A$ be a map defined by $\psi (ae\otimes eb)=aeb$.  It is easy to see that $\psi$ is a graded $A$-$A$ bimodule homomorphism with $C:=\Im \psi=AeA=(e)$ the two-sided ideal of $A$ generated by $e$.  Let $K:=\Ker \psi$ so that 
$$0\to K\to Ae\otimes _{eAe}eA\to A\to A/(e)\to 0$$ is an exact sequence in $\GrMod A^e$.  By applying $(-)e$ to the above exact sequence, we have 
$$0\to Ke\to Ae\otimes _{eAe}eAe \overset{\sim}{\rightarrow } Ae\to 0\to 0,$$ so $Ke=0$.  Since $A/(e)\in \Tors A$ and $K$ can be view as a graded right $A/(e)$-module, it follows that $K\in \Tors A$.  For $N\in \GrMod A$, exact sequences 
\begin{align*}
& 0\to K\to Ae\otimes _{eAe}eA\to C\to 0, \\
& 0\to C\to A\to A/(e)\to 0
\end{align*}
in $\GrMod A^e$ induce exact sequences  
\begin{align*}
& 0\to \uHom_A(C, N)\to \uHom_A(Ae\otimes _{eAe}eA, N)\to \uHom_A(K, N) \\
& 0\to \uHom_A(A/(e), N)\to \uHom_A(A, N)\to \uHom_A(C, N)\to \uExt^1_A(A/(e), N)
\end{align*}
in $\GrMod A$.  Since $A$ satisfies $\widetilde {\chi}_1$, we have that
$$\uHom_A(K, N),\; \uHom_A(A/(e), N),\; \uExt^1_A(A/(e), N)\in \Tors A,$$
so
\begin{align*}
\pi \uHom_{eAe}(Ae, \uHom_A(eA, N))&\cong \pi \uHom_A(Ae\otimes _{eAe}eA, N)\\
&\cong \pi \uHom_A(C, N)\cong \pi \uHom_A(A, N)\cong \pi N
\end{align*}
in $\Tails A$.


Since $\uHom_A(eA, -)\cong (-)e\cong -\otimes _AAe:\GrMod A\to
\GrMod eAe$, it follows that $\uHom_{eAe}(Ae, -)\cong -\otimes
_{eAe}eA:\Tails eAe\to \Tails A$ by Lemma \ref{lem.ample},
so if $Ae\in \grmod eAe$, then
$$\uHom_A(eA, -): \tails A\cong \tails eAe: \uHom_{eAe}(Ae, -)$$ are
equivalence functors quasi-inverse to each other by Lemma
\ref{lem.ample} again.
\end{proof}

In the second statment of the above lemma, we may replace the condition $\widetilde{\chi}_1$ by the condition $\chi_1$. 

\subsection{Ampleness of a Group Action} 

Following the previous subsections, we will define a notion of ampleness of a group action on a graded algebra.  
Let $A$ be a graded algebra, and $r\in \NN^+$.  
The $r$-th Veronese algebra of $A$ is defined by $A^{(r)}:=\bigoplus _{i\in \NN}A_{ri}$, and the $r$-th quasi-Veronese algebra of $A$ is defined by 
$$A^{[r]}:= \begin{pmatrix} A^{(r)} &
A(1)^{(r)} & \cdots &
A(r-1)^{(r)} \\
A(-1)^{(r)} & A^{(r)} & \cdots & A(r-2)^{(r)}
\\
\vdots & \vdots & \ddots & \vdots \\
A(-r+1)^{(r)} & A(-r+2)^{(r)} & \cdots &
A^{(r)}\end{pmatrix}$$
where the
multiplication of $A^{[r]}$ is given by
$(a_{ij})(b_{ij})=(\sum_ka_{kj}b_{ik})$ (see \cite {MM}).

\begin{theorem} \label{thm.1} 
Let $A$ be a right noetherian connected
graded algebra of $\depth A\geq 2$ satisfying
$\chi_1$, and $r\in \NN^+$. Then $(\cA, (r))$ is ample for $\tails
A$ if and only if 
$$(-)e:\tails A^{[r]}\to \tails e(A^{[r]})e$$
is an equivalence functor where $$e=\begin{pmatrix} 1
& 0 & \cdots &
0 \\ 0 & 0 & \cdots & 0 \\
\vdots & \vdots & \ddots & \vdots \\
0 & 0 & \cdots & 0 \end{pmatrix}\in \begin{pmatrix} A^{(r)} &
A(1)^{(r)} & \cdots &
A(r-1)^{(r)} \\
A(-1)^{(r)} & A^{(r)} & \cdots & A(r-2)^{(r)}
\\
\vdots & \vdots & \ddots & \vdots \\
A(-r+1)^{(r)} & A(-r+2)^{(r)} & \cdots &
A^{(r)}\end{pmatrix}=A^{[r]}$$ is an idempotent. 
\end{theorem}

\begin{proof} 
By \cite[Theorem 3.5]{Mmc}),  $(\cA, (r))$ is ample for $\tails A$ if and only if $(-)^{(r)}:\grmod A\to \grmod A^{(r)}$ induces an equivalence functor $(-)^{(r)}:\tails A\to \tails A^{(r)}$.  
There exists an
equivalence functor $Q:\grmod A\to \grmod A^{[r]}$ defined by
$$Q(M)=\begin{pmatrix} M^{(r)} \\ M(-1)^{(r)} \\ \vdots \\
M(-r+1)^{(r)}\end{pmatrix}$$
where the right action of $A^{[r]}$ on $Q(M)$ is
given by $(m_i)(a_{ij})=(\sum _km_ka_{ik})$ (cf. \cite [Remark 4.9]{MM}).  
Since 
$$e(A^{[r]})e=\begin{pmatrix} A^{(r)} & 0 & \cdots &
0 \\
0 & 0 & \cdots & 0
\\
\vdots & \vdots & \ddots & \vdots \\
0 & 0 & \cdots & 0\end{pmatrix}\cong A^{(r)}$$ as graded algebras and 
$$Q(M)e=\begin{pmatrix} M^{(r)} \\ M(-1)^{(r)} \\ \vdots \\
M(-r+1)^{(r)}\end{pmatrix}\begin{pmatrix} 1 & 0 & \cdots &
0 \\ 0 & 0 & \cdots & 0 \\
\vdots & \vdots & \ddots & \vdots \\
0 & 0 & \cdots & 0 \end{pmatrix}=\begin{pmatrix} M^{(r)} \\ 0 \\ \vdots \\ 0 \end{pmatrix}\cong
M^{(r)}
$$ in $\grmod A^{(r)}$,  $(-)^{(r)}:\tails A\to \tails A^{(r)}$ is an equivalence functor if and only if 
the composition of functors 
$$\begin{CD} (-)e:\tails A^{[r]} @>{Q^{-1}}>> \tails
A @>(-)^{(r)}>> \tails A^{(r)}\end{CD}$$ 
is an equivalence functor.  
\end{proof} 


Let $A=k\<x_1, \dots, x_n\>/I$ be a right noetherian connected graded algebra of $\depth A\geq 2$ satisfying $\chi_1$ 
and $G=\<\diag (\xi ^{\deg x_1}, \dots, \xi ^{\deg x_n})\>\leq \GL(n, k)$ a cyclic subgroup generated by a diagonal matrix where $\xi\in k^{\times}$ is a primitive $r$-th root of unity.  
By \cite[Theorem 4.4 (1)]{Mmc}, $A^{[r]}\cong A*G$ as ungraded algebras, and in this isomorphism, it is shown that the idempotent $e\in A^{[r]}$ above is sent to $\frac{1}{|G|}\sum 1*g\in A*G$.  Since $e(A^{[r]})e\cong A^{(r)}\cong A^G$ as ungraded algebras, we define the notion of ampleness of a group action motivated by the above result.  


\begin{definition}
Let $A$ be a right noetherian graded algebra, $G\leq \GrAut A$
a finite subgroup such that $\fchar k$ does not divide $|G|$, and $e=\frac{1}{|G|}\sum 1*g\in A*G$. We say that
$G$ is ample for $A$ if
$$
(-)e
:\tails A*G\to \tails A^G$$ is an
equivalence functor.
\end{definition}

Note that, in the above setting, $A*G$ is right noetherian.
Moreover, it follows from \cite[Corollary 1.12]{Mfs} that $A^G$ is also right noetherian.

\begin{lemma} \label{lem.amis} 
Let $A$ be a right noetherian graded algebra of finite global dimension.   If $G\leq \GrAut A$ is a finite ample subgroup such that $\fchar k$ does not divide $|G|$, then $A^G$ is a graded isolated singularity.  
\end{lemma} 

\begin{proof} 
Since $A$ has finite global dimension,
$$\gldim (\tails A^G)=\gldim (\tails A*G)\leq \gldim A*G=\gldim
A<\infty$$ by \cite [Theorem 7.5.6]{MR}, so $A^G$ is a graded isolated
singularity.
\end{proof} 

\begin{theorem} \label{thm.ample}
Let $A$ be a noetherian graded algebra, $G\leq \GrAut A$ a finite
subgroup such that $\fchar k$ does not divide $|G|$, and $e:=\frac{1}{|G|}\sum 1*g\in A*G$. If $A*G/(e)$ is
finite dimensional, then $G$ is ample for $A$.  In fact, 
$$-\otimes _{A*G}(A*G)e : \tails A*G\rightleftarrows \tails A^G:-\otimes _{A^G}e(A*G)$$
are equivalence functors quasi-inverse to each other.
Moreover, if $A$ is a locally finite $\NN$-graded algebra satisfying $\chi_1$, then
$$\uHom_{A*G}(A, -): \tails A*G\rightleftarrows \tails A^G: \uHom_{A^G}(A, -)$$
are equivalence functors quasi-inverse to each other.
\end{theorem}

\begin{proof}
Since $A$ is noetherian, both $A*G$ and $A^G$ are noetherian.
Moreover, $(A*G)e\cong A$ is finitely generated as graded right $A^G$-module and 
$e(A*G)\cong A$ is finitely generated as graded left $A^G$-module by Lemma \ref{lem.right} and \cite[Corollary 5.9]{Mfs}, so 
$$-\otimes _{A*G}(A*G)e
: \tails A*G\rightleftarrows \tails A^G:-\otimes _{A^G}e(A*G)$$ are equivalence
functors quasi-inverse to each other by Lemma \ref{lem.ample}.

Suppose that $A$ is a locally finite $\NN$-graded algebra satisfying $\chi_1$. Recall that we identify $A=(A*G)e$ in $\grmod A^G$, and
$A=e(A*G)$ in $\grmod A*G$.  
Since $A$ satisfies $\chi_1$ and $A*G$ is finitely generated as a graded left and right $A$-module, $A*G$ also satisfies $\chi_1$ by \cite[Theorem 8.3 (1)]{AZ} (cf. \cite[Proposition 3.1 (3), Corollary 3.6 (1), Proposition 3.11 (2)]{AZ}), 
so
$$\uHom_{A*G}(A, -): \tails A*G\rightleftarrows \tails A^G: \uHom_{A^G}(A, -)$$
are equivalence functors quasi-inverse to each other by Lemma
\ref{lem.ample2}.
\end{proof}

\section{Graded Skew Bimodule Calabi-Yau Algebras}

Graded bimodule Calabi-Yau algebras are an important class of algebras studied in representation theory of algebras.  In this section, we will prove that a graded skew bimodule Calabi-Yau algebra is an AS-regular algebra over $R$ defined in \cite {MM}, which is one of the generalizations of an AS-regular algebra.  

\subsection{Conventions on Bimodule Structures} 

In this subsection, we will fix some conventions on bimodule structures.  Although results in this subsection are known to the experts, it is sensitive to check the bimodule structures in each isomorphism in the proof of Theorem 3.5, so we will give proofs in some detail.  

For a graded algebra $A$, we define the enveloping algebra of $A$ by $A^e:=A^o\otimes _kA$.  The graded $A$-$A$ bimodule
$X$ can be viewed as a graded right $A^e$-module by $x(a\otimes
b):=axb$. By this right action, we have a commutative diagram
$$\begin{CD}
A^e\times A^e @>{\textnormal{ product }}>> A^e \\
\parallel & & \parallel \\
(A\otimes _kA)\times A^e @>{\textnormal{ right action }}>> A\otimes
_kA,
\end{CD}$$
so we can identify $A^e$ with $A\otimes _kA$ as a graded right
$A^e$-module. The map $\varphi:A^e\to (A^e)^o; \; a\otimes b\to
b\otimes a$ is an isomorphism of graded algebras 
so the graded $A$-$A$ bimodule $X$ can be viewed as a graded left
$A^e$-module via $\varphi $, that is, $(a\otimes
b)x:=x\varphi(a\otimes b)=x(b\otimes a)=bxa$.
For $X, Y\in \GrMod A^e$, we define the graded $A^e$-$A^e$ bimodule
structure on $X\otimes _kY$ by
$$(a\otimes b)(x\otimes y)(c\otimes d):=cxa\otimes byd.$$

For an abelian category $\cC$, we denote by $\cD(\cC)$ the derived category of $\cC$, and by $\cD^b(\cC)$ the bounded derived category of $\cC$.  

\begin{lemma} \label{lem.bibi}
Let $A$ be a graded algebra.  For $X, Y, Z\in \GrMod A^e$,
\begin{enumerate}
\item{}
$Z\lotimes _{A^e}(X\otimes _kY)\cong X\lotimes _AZ\lotimes _AY$
\item{} $(X\otimes _kY)\lotimes_{A^e}Z\cong Y\lotimes _AZ\lotimes
_AX$
\end{enumerate}
in $\cD(\GrMod A^e)$.
\end{lemma}

\begin{proof}
(1) It is easy to check that the map 
$$\phi:Z\otimes _{A^e}(X\otimes _kY)\to X\otimes _AZ\otimes
_AY$$ defined by $\phi(z\otimes x\otimes y)=x\otimes z\otimes y$ is
a well-defined 
homomorphism in $\GrMod A^e$.
If $Z=A^e$, then $\phi$ induces an isomorphism of graded vector
spaces
$$\begin{array}{ccccccc}
A^e\otimes _{A^e}(X\otimes _kY) & \cong &  X\otimes _kY & \cong &
X\otimes _AA\otimes _kA\otimes _AY & \cong & X\otimes _AA^e\otimes
_AY \\
(a\otimes b)\otimes (x\otimes y) & \rightarrow & xa\otimes by &
\leftarrow & (x\otimes a)\otimes (b\otimes y) & \leftarrow &
x\otimes (a\otimes b)\otimes y,
\end{array}$$
so $\phi$ is an isomorphism in $\GrMod A^e$. Since tensor products
commute with arbitrary direct sum and degree shift, if $Z$ is free
in $\GrMod A^e$, then $\phi$ is an isomorphism in $\GrMod A^e$.  If
$F$ is a free resolution of $Z$ in $\GrMod A^e$, then $F$ is
simultaneously a free resolution of $Z$ in $\GrMod A$ and $\GrMod
A^o$, so $$Z\lotimes _{A^e}(X\otimes _kY)\cong F\otimes
_{A^e}(X\otimes _kY)\cong X\otimes _AF\otimes _AY\cong X\lotimes
_AZ\lotimes _AY$$ in $\cD(\GrMod A^e)$.

(2) The proof is similar to (1).  
\end{proof}

For $X, Y\in \GrMod A^e$, we define the graded $A^e$-$A^e$ bimodule
structure on $\uHom_k(X, Y)$ by
$$\left ( (a\otimes b)\varphi(c\otimes d)\right ) (x)=b\varphi (axc)d.$$ 


\begin{lemma} \label{lem.bibi2}
Let $A$ be a graded algebra.  For $X, Y\in \GrMod A^e$,
$$\RuHom_A(X\otimes
_kA, Y)\cong \uHom_A(X\otimes _kA, Y)\cong \uHom_k(X, Y)\cong
DX\otimes _kY$$ in $\cD(\GrMod (A^e)^e)$ (in the derived category of
graded $A^e$-$A^e$ bimodules).
\end{lemma}

\begin{proof}
Viewing $X\otimes _kA, Y$ as graded $A$-modules,
\begin{align*}
\RuHom _A(X\otimes _kA, Y) & \cong \RuHom_A(X\lotimes _kA, Y) \\
& \cong \RuHom_k(X, \RuHom_A(A, Y)) \\
& \cong 
\uHom _k(X, Y)
\end{align*} in $\cD(\GrMod k)$.  It follows that $\RuHom_A(X\otimes _kA, Y)$ is concentrated in complex degree 0, so $\RuHom_A(X\otimes _kA, Y)\cong \uHom_A(X\otimes _kA, Y)$ in $\cD(\GrMod (A^e)^e)$.  It is easy to check that the map
$$\a :\uHom _A(X\otimes _kA, Y) \to \uHom_k(X, \uHom_A(A, Y))\to \uHom _k(X, Y)$$ defined by $\a(\phi)(x)=\phi
(x\otimes 1)$ is 
an isomorphism as graded $A^e$-$A^e$-bimodules.
On the other hand, 
it is easy to check that the map
$$\b:DX\otimes _kY=\uHom_k(X, k)\otimes _kY\to \uHom _k(X, k\otimes
_kY)\to  \uHom _k(X, Y)$$ defined by $\b(\psi\otimes y)(x)=\psi
(x)y$ is 
an isomorphism of graded $A^e$-$A^e$ bimodules. 
\end{proof}

\subsection{Graded Skew Bimodule Calabi-Yau Algebras} 

AS-regular algebras are important class of algebras studied in noncommutative algebraic geometry, while Calabi-Yau algebras are important class of algebras studied in representation theory.  In this subsection, after extending the notions of these two classes of algebras, we will see some relationships between them.  

For a graded algebra $A$, we denote by $\operatorname{perf} ^{\ZZ}A$ the full subcategory of $\cD^b(\grmod A)$ consisting of complexes quasi-isomorphic to perfect complexes, that is, complexes of finite length whose terms are finitely generated projectives.  Extending the notion of a Calabi-Yau algebra, the following class of algebras has been defined and studied in several papers (see \cite {RRZ}).   

\begin{definition} 
A graded algebra $A$ is called graded skew bimodule Calabi-Yau of
dimension $d$ and of Gorenstein parameter $\ell$ if $A\in
\operatorname {perf} ^{\ZZ}A^e$ and $\RuHom_{A^e}(A, A^e)\cong
{_{\nu}A}(\ell)[-d]$ in $\cD(\GrMod A^e)$ for some $\nu \in
\GrAut A$ called the Nakayama automorphism of $A$.
\end{definition}

There are a few generalizations of the notion of AS-regularity in the literature.  The following generalization is used in this paper.  

\begin{definition} \label{def.2}
(\cite [Definition 3.1]{MM})
A locally finite $\NN$-graded algebra $S$ with $R=S_0$, is called
AS-regular over $R$ of dimension $d$ and Gorenstein parameter
$\ell$ if
\begin{enumerate}
\item{} $\gldim R<\infty$,
\item{} $\gldim S=d<\infty$, and
\item{} $\RuHom_S(R, S)\cong {_{\s}DR}(\ell)[-d]$ in $\cD(\GrMod R^e)$ for some $\s\in
\Aut_kR$.
\end{enumerate}
\end{definition}

AS-regularity as defined in Definition \ref{def.as} is the same as AS-regularity over $k$ as defined above.
The above two classes of algebras are closely related.
In fact, they are equivalent if $A$ is connected graded by 
\cite [Lemma 1.2]{RRZ}.
We will show one implication for a non-connected case. 

\begin{theorem} \label{thm.sbcy}
Let $A$ be a locally finite $\NN$-graded skew bimodule Calabi-Yau
algebra with $R=A_0$ of dimension $d$ and of Gorenstein parameter
$\ell$. If $\gldim R<\infty$ and $\gldim A=d$, then $A$ is an
AS-regular algebra over $R$ of dimension $d$ and of Gorenstein
parameter $\ell$.
\end{theorem}

\begin{proof}
It is enough to check condition (3) in the above definition.
Since $A$ is a graded skew bimodule Calabi-Yau algebra of dimension
$d$ and of Gorenstein parameter $\ell$,
we see that $\RuHom_{A^e}(A, A^e)\cong {_{\nu}A}(\ell)[-d]$ in $\cD(\GrMod A^e)$
where $\nu \in \GrAut A$ is the Nakayama
automorphism of $A$. Since
$$
\begin{array}{lll}
\RuHom_A(R, A)& \cong \RuHom_A(R\lotimes _AA\lotimes _AA, A) \\
& \cong \RuHom_A(A\lotimes _{A^e}(R\otimes _kA), A) & \qquad \textnormal {[Lemma \ref{lem.bibi} (1)]} \\
& \cong \RuHom_{A^e}(A, \RuHom_A(R\otimes _kA, A)) \\
& \cong \RuHom_{A^e}(A, DR\otimes _kA) & \qquad \textnormal {[Lemma \ref{lem.bibi2}]} \\
& \cong \RuHom_{A^e}(A, (DR\otimes _kA)\lotimes _{A^e}A^e) \\
& \cong (DR\otimes _kA)\lotimes _{A^e}\RuHom_{A^e}(A, A^e) & \qquad \textnormal {[$A\in \operatorname{perf}^{\ZZ}A^e$]}\\
& \cong (DR\otimes _kA)\lotimes _{A^e}{_{\nu}A}(\ell)[-d] \\
& \cong A\lotimes _A{_{\nu}A}\lotimes _ADR(\ell)[-d] & \qquad \textnormal {[Lemma \ref{lem.bibi} (2)]}\\
& \cong {_{\nu}DR}(\ell)[-d],
\end{array}
$$
in $\cD(\GrMod A^e)$ (in $\cD^b(\GrMod R^e)$), $A$ is an AS-regular
algebra over $R$ of dimension $d$ and of Gorenstein parameter
$\ell$.
\end{proof}

\begin{corollary} \label{cor.AS}
If $S$ is a noetherian AS-regular algebra over $k$ of dimension
$d$ and of Gorenstein parameter $\ell$, and $G\leq \GrAut S$
is a finite subgroup such that $\fchar k$ does not
divide $|G|$, then $S*G$ is a noetherian AS-regular algebra over
$kG$ of dimension $d$ and of Gorenstein parameter $\ell$.
\end{corollary}

\begin{proof}
Since $S$ is locally finite $\NN$-graded and $G$ is finite, $S*G$ is
locally finite $\NN$-graded.  Since $\fchar k$ does not divide $|G|$, we see that $(S*G)_0=kG$ is
semi-simple, that is, $\gldim kG=0$.  Since $S$ is connected graded AS-regular and $kG$ is semisimple, $S*G$ is graded skew bimodule Calabi-Yau of dimension $d$ and of Gorenstein parameter $\ell$ by (the proof of) \cite [Theorem 4.1]{RRZ}.  
Since $S$ is noetherian, 
$S*G$ is also
noetherian. By \cite [Theorem 7.5.6]{MR}, $\gldim S*G=\gldim S=d$.
By Theorem \ref{thm.sbcy},
$S*G$ is an AS-regular algebra over $kG$
of dimension $d$ and of Gorenstein parameter $\ell$.
\end{proof}

\begin{theorem} \label{thm.**}
Let $S$ be a noetherian AS-regular algebra over $k$ of dimension
$d\geq 2$, and $G\leq \GrAut S$ a finite ample subgroup. If
$\fchar k$ does not divide $|G|$, then the map
$$S * G \rightarrow \uEnd_{S^G}(S); \quad s * g \mapsto [t \mapsto s g(t)]$$
is an isomorphism of graded algebras.
\end{theorem}

\begin{proof}
Since $S\in \grmod S^G$, it follows that $\depth _{S^G}S=\depth
_SS=d\geq 2$.  Since $S^G$ is a noetherian connected graded algebra,
$$\uEnd_{S^G}(S)=B(\grmod S^G, S,
(1))\cong B(\tails S^G, \pi S, (1))$$ as graded algebras by \cite
[Lemma 3.3]{Mmc}.  Since $S*G$ is a noetherian AS-regular algebra
over $kG$ of dimension $d\geq 2$ by Corollary \ref {cor.AS},
$$B(\tails (S*G), \pi (S*G), (1))\cong S*G$$ as graded algebras by
\cite [Proposition 4.4]{MM} (cf. \cite [Remark 4.5]{MM}).  By Lemma \ref{lem.right},
$$S*G\otimes_{S*G}(S*G)e\cong (S*G)e\cong S$$
in $\grmod S^G$.  Since $G$ is
ample for $S$, the equivalence functor
$$-\otimes _{S*G}(S*G)e: \tails (S*G)\to \tails (S^G)$$
induces an isomorphism of algebraic triples
$$(\tails (S*G), \pi (S*G),
(1))\to (\tails (S^G), \pi S, (1)),$$ hence
$$S*G\cong B(\tails (S*G), \pi (S*G),
(1))\cong B(\tails S^G, \pi S, (1))\cong \uEnd_{S^G}(S).$$
as graded algebras.  It is easy to see that the composition of the above isomorphisms $\Phi:S*G\to \uEnd_{S^G}(S)$ sends an element $s*g\in S*G$ to the right multiplication of $(s*g)$ on $S$ via the identification $(\star)\; S\to (S*G)e; \; t\mapsto (t*1)e$ in Lemma \ref{lem.right}, so 
\begin{align*}
\Phi(s * g)(t) 
& =^{(\star)}(s*g)(t*1)e 
=(sg(t)*g)\left ( \frac{1}{|G|}\sum _{h\in G}1*h\right ) \\
& =\frac{1}{|G|}\sum _{h\in G}sg(t)*gh=\frac{1}{|G|}\sum _{h\in G}sg(t)*h \\
& =(sg(t)*1)\left ( \frac{1}{|G|}\sum _{h\in G}1*h\right )=(sg(t)*1)e \\
& =^{(\star)}sg(t).  
\end{align*}
\end{proof}

\begin{lemma} \label{lem.stabhom}
Let $A$ be a noetherian connected graded algebra.
For $M, N \in \grmod A$,
there exists an exact sequence
$$\begin{array}{ccccccc}
N \otimes_A \uHom_A(M,A)
&\overset{\rho_{M}^{N}}{\rightarrow} 
&\uHom_A(M,N)
&\rightarrow 
&\uHom_A(M,N)/P(M,N)
&\rightarrow 
&0
\\
n \otimes \phi
&\mapsto 
&[m \mapsto  n\phi(m)]
\end{array}
$$
where $P(M,N)$ is a graded subgroup of $\uHom_A(M,N)$ consisting of morphisms factoring through a finitely generated projective graded $A$-module (i.e. free).
\end{lemma}

\begin{proof}
The proof is analogous to the proof in the ungraded commutative case, as given in \cite[Lemma 3.8]{Y}.
\end{proof}

\begin{theorem} \label{thm.fdim}
Let $A$ be a noetherian graded algebra of finite global dimension.
Assume that there exists
an idempotent $e \in A$ such that
$eAe$ is a noetherian AS-Gorenstein algebra over $k$ of dimension $d\geq 2$ and $Ae \in \CM^{\ZZ}(eAe)$.
If $eAe$ is a graded isolated singularity, and the graded algebra homomorphism
$$\Phi: A \rightarrow  \uEnd_{eAe}(Ae); \quad a \mapsto [xe \mapsto axe]$$
is an isomorphism, 
then $A/(e)$ is finite dimensional over $k$.
\end{theorem}

\begin{proof}
Let $\psi: Ae \otimes_{eAe} eA \to A$ be a homomorphism defined by $\psi(ae\otimes eb)= aeb$.
Since the composition of the following homomorphisms
\[ \begin{array}{rlrl}
A &\cong \uHom_A(A,A)             & a & \mapsto [x \mapsto ax] \\
  &\overset{\uHom(\psi ,A)}{\longrightarrow } \uHom_A(Ae \otimes_{eAe} eA,A)    &   & \mapsto [xe\otimes ey \mapsto axey] \\
  &\cong \uHom_{eAe}(Ae, \uHom_A(eA,A))   &   & \mapsto [xe \mapsto [ey \mapsto axey]] \\
  &\cong \uHom_{eAe}(Ae, Ae)= \uEnd_{eAe}(Ae)  &   & \mapsto [xe \mapsto axe] 
\end{array} \]
is equal to $\Phi$, we see that $\uHom(\psi ,A)$ is an isomorphism by our assumption.
Let $\e: A \to eA$ be a split epimorphism defined by $\e(a)=ea$. Then the following commutative diagram
\begin{align*}
\xymatrix@C=4pc@R=1.5pc{
&\uHom_A(A,A) \ar[r]^(0.45){\uHom(\psi ,A)}_(0.42){\cong}  \ar@{->>}[d]_{\uHom(A ,\e)}
&\uHom_A(Ae \otimes_{eAe} eA,A) 
\ar@{->>}[d]^{\uHom(Ae\otimes eA ,\e)}
\\
&\uHom_{A}(A, eA) \ar[r]^(0.45){\uHom(\psi ,eA)}
&\uHom_A(Ae \otimes_{eAe} eA,eA)  
}
\end{align*}
implies that $\uHom(\psi ,eA)$ is surjective.
Let $\Psi$ be the composition of the following homomorphisms
\[ \begin{array}{rlrl}
eA &\cong \uHom_A(A,eA)             & eb & \mapsto [x \mapsto ebx] \\
  &\overset{\uHom(\psi,eA)}{\longrightarrow } \uHom_A(Ae \otimes_{eAe} eA,eA)    &   & \mapsto [xe\otimes ey \mapsto ebxey] \\
  &\cong \uHom_{eAe}(Ae, \uHom_A(eA,eA))   &   & \mapsto [xe \mapsto [ey \mapsto ebxey]] \\
  &\cong \uHom_{eAe}(Ae, eAe)  &   & \mapsto [xe \mapsto ebxe].
\end{array}\]
Since $\uHom(\psi,eA)$ is surjective, $\Psi$ is surjective, so 
$$\Phi':=Ae\otimes _{eAe}\Psi:Ae\otimes _{eAe}eA\to Ae\otimes _{eAe}\uHom_{eAe}(Ae, eAe)$$ is surjective.
By Lemma \ref{lem.stabhom}, we have a commutative diagram
\begin{align*}
\xymatrix@C=2pc@R=1pc{
Ae \otimes_{eAe} eA \ar[r]^{\psi}  \ar@{->>}[d]^{\Phi'} 
&A \ar[r] \ar[d]_{\cong}^{\Phi} \ar[r]
&A/(e) \ar[r] \ar@{.>}[d]
&0\\
Ae \otimes_{eAe} \uHom_{eAe}(Ae, eAe) \ar[r]^(0.6){\rho_{Ae}^{Ae}}
&\uEnd_{eAe}(Ae) \ar[r]
&\frac{\uEnd_{eAe}(Ae)}{P(Ae,Ae)} \ar[r]
&0}
\end{align*}
with exact rows.
By the five lemma, the induced map $A/(e) \to \uEnd_{eAe}(Ae)/P(Ae,Ae)$ is injective.
Since $eAe$ is a noetherian AS-Gorenstein graded isolated singularity of dimension $d \geq 2$,
we have
\begin{align*}
\uEnd_{eAe}(Ae)/P(Ae,Ae)
\cong \uEnd_{\uCM^{\ZZ}{eAe}}(Ae)
\cong D\uExt^1_{eAe}(Ae, \tau Ae)
\end{align*}
by \cite[Corollary 4.6]{U} where $\uCM^{\ZZ}(A)$ is the stable category of $\CM^{\ZZ} (A)$
and $\t$ is the Auslander-Reiten translation on $\uCM^{\ZZ} (A)$.
By \cite[Lemma 5.7]{U}, $\uExt^1_{eAe}(Ae, \tau Ae)$ is finite dimensional over $k$, so $A/(e)$ is finite dimensional over $k$.
\end{proof}

\begin{theorem} \label{thm.ampeq}
Let $S$ be a noetherian AS-regular algebra over $k$ of dimension $d\geq 2$
and $G\leq \HSL(S)$ a finite subgroup 
such that 
$\fchar k$ does not divide $|G|$.
Then the following are equivalent.
\begin{enumerate}
\item $G$ is ample for $S$.
\item $S^G$ is a graded isolated singularity, and 
$$\Phi: S * G \rightarrow \uEnd_{S^G}(S); \quad s * g \mapsto [t \mapsto s g(t)]$$
is an isomorphism of graded algebras.
\item $S*G/(e)$ is finite dimensional over $k$ where $e = \frac{1}{|G|}\sum_{g\in G}1*g \in S*G$.
\end{enumerate}
\end{theorem}

\begin{proof}
(1) $\Rightarrow $ (2): This follows from Lemma \ref{lem.amis} and Theorem \ref{thm.**}. 

(2) $\Rightarrow $ (3):  Since $S*G$ is a noetherian graded algebra of finite global dimension, $e(S*G)e\cong S^G$ is a noetherian AS-Gorenstein algebra over $k$ of dimension $d\geq 2$ by Lemma \ref{lem.eGe} and Theorem \ref{thm.asg}, and $(S*G)e\cong S\in \CM^{\ZZ}(S^G)$ by Lemma \ref{lem.right}, this follows from Theorem \ref{thm.fdim}. 

(3) $\Rightarrow $ (1): This follows from Theorem \ref{thm.ample}. 
\end{proof}

\begin{corollary} \label{cor.ampeq}
Let $S=k[x_1, \dots, x_d]$ be a commutative polynomial algebra 
such that $\fchar k =0$, $\deg x_i=1$ and $d\geq 2$,
and let $G\leq \SL(d, k)$ be a finite subgroup.
Then the following are equivalent.
\begin{enumerate}
\item $G$ is ample.
\item $S^G$ is an isolated singularity.
\item $S*G/(e)$ is finite dimensional over $k$
 where $e = \frac{1}{|G|}\sum_{g\in G}1*g \in S*G$.
\item $G$ acts freely on $\AA^n\setminus \{0\}$.
\end{enumerate}
\end{corollary}

\begin{proof} Since $G\leq \SL(d, k)$, it follows that $G$ is small, so
$$\Phi: S * G \rightarrow \uEnd_{S^G}(S); \quad s * g \mapsto [t \mapsto s g(t)]$$ is an isomorphism of graded algebras by \cite [Theorem 3.2]{IT}.  Since $\SL(d, k)=\HSL(S)$, the equivalences (1) $\Leftrightarrow$ (2) $\Leftrightarrow $ (3) follows from Theorem \ref{thm.ampeq}.  
The equivalence (2) $\Leftrightarrow$ (4) follows from \cite [8.2]{IY}.  
\end{proof} 

\subsection{Beilinson Algebras} 

The following is a noncommutative generalization of a Beilinson algebra.  

\begin{definition} (\cite [Definition 4.7]{MM}) 
The Beilinson algebra of an AS-regular algebra over $R$ of
Gorenstein parameter $\ell$ is defined by
$$\nabla S:=\End _S\left(\bigoplus _{i=0}^{\ell-1}S(i)\right)=\begin{pmatrix} S_0 & S_1 & \cdots & S_{\ell-1} \\
0 & S_0 & \cdots & S_{\ell-2} \\
\vdots & \vdots & \ddots & \vdots \\
0 & 0 & \cdots & S_0\end{pmatrix}$$ with the multiplication
$(a_{ij})(b_{ij})=\left (\sum _{k=0}^{\ell-1}a_{kj}b_{ik}\right )$. (We impose this multiplication formula so that $\nabla S=\left(S^{[\ell]}\right)_0$.)  
\end{definition}

Note that if a group $G$ acts on an AS-regular algebra $S$, then $G$
naturally acts on $\nabla S$ by $g((a_{ij})):=(g(a_{ij}))$.  It is easy to check the following result. 

\begin{lemma} \label{lem.impo}
Let $S$ be a noetherian AS-regular algebra over $k$ of Gorenstein
parameter $\ell$ and $G\leq \GrAut S$ a finite subgroup. If
$\fchar k$ does not divide $|G|$, then $(\nabla
S)*G\cong \nabla (S*G)$ as algebras.
\end{lemma}


\begin{theorem} \label{thm.tilt}
Let $S$ be a noetherian AS-regular algebra over $k$ of dimension
$d\geq 2$ and of Gorenstein parameter $\ell$, and $G\leq
\GrAut S$ a finite ample subgroup. If $\fchar k$
does not divide $|G|$, then
$$\cD^b(\tails S^G)\cong \cD^b(\mod (\nabla S)*G)\cong \cD^b\left(\mod \End _{S^G}\left(\bigoplus
_{i=0}^{\ell-1}S(i)\right)\right).$$ Moreover, $\bigoplus _{i=0}^{\ell-1}\pi S(i)$
is a tilting object in $\cD^b(\tails S^G)$.
\end{theorem}

\begin{proof}
Since $G$ is ample and $S*G$ is a noetherian AS-regular algebra over
$kG$ of dimension $d\geq 2$ and of Gorenstein parameter $\ell$ by
Corollary \ref{cor.AS},
$$\cD^b(\tails S^G)\cong \cD^b(\tails (S*G)) \cong \cD^b(\mod \nabla (S*G))\cong \cD^b(\mod
(\nabla S)*G)$$ by \cite [Theorem 4.14]{MM} and Lemma
\ref{lem.impo}. Since the equivalence functor
$$(-)e:\tails (S*G)\to \tails S^G$$
sends $\pi (S*G)(i)$ to $\pi (S*G)e(i)\cong \pi S(i)$
for $i\in \ZZ$ by Lemma \ref{lem.right}, and we have $\depth _{S^G}(\bigoplus
_{i=0}^{\ell-1}S(i))=d\geq 2$,
\begin{align*}
\nabla (S*G)  
& \cong
\End _{\tails (S*G)}\left(\bigoplus _{i=0}^{\ell-1}\pi (S*G)(i)\right) \\
& \cong \End _{\tails S^G}\left(\bigoplus _{i=0}^{\ell-1}\pi S(i)\right) \\
& \cong
\End _{S^G}\left(\bigoplus _{i=0}^{\ell-1}S(i)\right)
\end{align*}
as algebras by \cite [Lemma 4.11]{MM} and \cite [Lemma
2.9]{Mmc}.  Since $\bigoplus _{i=0}^{\ell-1}\pi (S*G)(i)$ is a tilting
object in $\cD^b(\tails S*G)$ by \cite [Proposition 4.3,
Proposition 4.4]{MM}, $\bigoplus _{i=0}^{\ell-1}\pi S(i)$ is a tilting
object in $\cD^b(\tails S^G)$.
\end{proof}

Let $S=k[x, y]$ be a polynomial algebra and $G\leq \SL(2, k)$ a finite subgroup. 
The classical McKay correspondence claims that
$$\cD^b\left(\widetilde {\Spec S^G}\right)\cong \cD^b(\mod S*G)$$ where $\widetilde {\Spec S^G}$ is the minimal resolution of $\Spec S^G$.
On the other hand, if $S$ is a noetherian AS-regular algebra over $k$ of dimension $d\geq 2$ and $G\leq \GrAut S$ is a finite ample subgroup, then $S^G$ is a graded isolated singularity by Lemma \ref{lem.amis}.  This tells that $\Proj S^G$ is smooth so that $\widetilde {\Proj S^G}=\Proj S^G$.
In this case, 
$$\cD^b\left(\widetilde {\Proj S^G}\right)=\cD^b(\Proj S^G)=\cD^b(\tails S^G)\cong \cD^b(\mod \nabla S*G)$$ by Theorem \ref{thm.tilt}, which is similar to McKay correspondence (see \cite {Mmc} for more details).  

It is interesting to compare with the following theorem.  

\begin{theorem} \label{thm.***}
Let $S$ be a noetherian AS-regular algebra over $k$ of dimension
$d\geq 2$, and $G\leq \HSL(S)$ a finite ample subgroup. If
$\fchar k$ does not divide $|G|$, then $S$ is a $(d-1)$-cluster tilting object in
$\CM^{\ZZ}(S^G)$.
\end{theorem}

\begin{proof}
Since $G$ is ample for $S$, we see that $S^G$ is a
graded isolated singularity by Lemma \ref{lem.amis}.  By Theorem \ref{thm.**}, $\uEnd_{S^G}(S)\cong S*G$ is a
finitely generated graded free module over $S$, so $S$ is a
$(d-1)$-cluster tilting object in $\CM^{\ZZ}(S^G)$ by \cite [Theorem
5.9]{U}.
\end{proof}

\section{Noetherian AS-regular Algebras of Dimension 2}

Let $S$ be a noetherian AS-regular algebra of dimension 2 and $G\leq \HSL(S)$ a finite subgroup.  The purpose of this last section is to find a quiver $Q_{S, G}$ such that $\cD^b(\tails S^G)\cong \cD^b(\mod kQ_{S, G})$.  Throughout this section, we assume that $k$ is an algebraically closed field of characteristic 0.  

\subsection{Classification of Group Actions}  In this subsection, we will show that, for a noetherian AS-regular algebra $S$ of dimension 2, any finite subgroup $G\leq \HSL(S)$ is ample unless $S\cong k\<x, y\>/(xy\pm yx)$.  

\begin{lemma} \label{lem.hdet}
Let $S$ be a noetherian AS-regular algebra and $\s \in \GrAut S $.  
\begin{enumerate}
\item If $x \in S$ is a
regular normal element such that $\s(x) = ax$ for some $a\in k^\times$, then 
$\hdet_S \s = a (\hdet_A \s)$ where $A = S/(x)$.
\item If $S$ is Koszul and $\s|_{S_1} = (c_{ij})$, then $\s^{t}|_{S_1^!} = (c_{ji})$
defines an automorphism $\s^{t} \in \GrAut S^!$, and
$ \s^{t}(u) = (\hdet \s) u $
for any $u \in S^!_d$.
\end{enumerate}
\end{lemma}

\begin{proof}
(1) is \cite[Proposition 2.4]{JiZ2}.
(2) is a special case of \cite [Lemma 4.6]{JZ}.
\end{proof}

It is known that if $S$ is a noetherian AS-regular algebra of dimension 2, then $S\cong k\<x, y\>/(f)$ for some homogeneous element $f\in k\<x, y\>$ where $\deg x=m, \deg y=n\in \NN^+$.  If $\gcd(m, n)=r\neq 1$, then $S^{(r)}\cong  k\<x, y\>/(f)$ with $\deg x=m/r, \deg y=n/r$ so that $\gcd(\deg x, \deg y)=1$, and $\tails S^G\cong (\tails (S^{(r)})^G)^{\times r}$ for any $G\leq \GrAut S=\GrAut S^{(r)}$, so it is enough to study the case that $\gcd(\deg x, \deg y)=1$ in understanding the category $\tails S^G$.   

\begin{lemma} \label{lem.Ue}
Let $S=k\<x, y\>/(f)$ be a noetherian AS-regular algebra of dimension 2 such
that $\gcd(\deg x, \deg y)=1$.  The following is a list of all
possible $f$ (up to isomorphism of graded algebras), $\s \in \GrAut S$ and corresponding $\hdet \s$.

$$\begin{array}{|c|c|c|c|}
\hline (\deg x, \deg y) & f & \s & \hdet \s \\
\hline
(1, 1) & xy-yx & \begin{matrix} x\mapsto ax+by \\
y\mapsto cx+dy \end{matrix} & ad-bc \\
\hline
(1, 1) & xy+yx & \begin{matrix} x\mapsto ax \\
y\mapsto dy \end{matrix} & ad \\
\hline
(1, 1) & xy+yx & \begin{matrix} x\mapsto by \\
y\mapsto cx \end{matrix} & bc \\
\hline
(1, 1) & xy-\a yx, \a\neq 0, \pm 1 & \begin{matrix} x\mapsto ax \\
y\mapsto dy \end{matrix} & ad \\
\hline
(1, q), q\geq 2 & xy-yx & \begin{matrix} x\mapsto ax \\
y\mapsto cx^q+dy \end{matrix} & ad \\
\hline
(1, q), q\geq 2 & xy-\a yx, \a \neq 0, 1 & \begin{matrix} x\mapsto ax \\
y\mapsto dy \end{matrix} & ad \\
\hline
(p, q), p, q\geq 2 & xy-\a yx, \a\neq 0 & \begin{matrix} x\mapsto ax \\
y\mapsto dy \end{matrix} & ad \\
\hline
(1, q), q\geq 1 & xy-yx-x^{q+1} & \begin{matrix} x\mapsto ax \\
y\mapsto cx^q+a^qy \end{matrix} & a^{q+1} \\
\hline
\end{array}$$
\end{lemma}

\begin{proof}
Since $S$ is isomorphic to a graded Ore extension of the polynomial algebra in one variable by \cite [Proposition 3.3]{SZ},
the list of $f$ follows from a calculation.
In each case, it is easy to find all graded algebra automorphisms of $S$.
If $\deg x=\deg y=1$, then $S$ is Koszul, so we can use Lemma \ref{lem.hdet} (2) to compute their homological determinants.
Otherwise, $\s(x)=ax$ for some $a\in k^\times$, so we can use Lemma \ref{lem.hdet} (1) to compute their homological determinants. 
\end{proof}

\begin{lemma} \label{lem.hsl}
Let  
$S=k\<x, y\>/(f)$ be a noetherian AS-regular algebra of dimension 2 such
that $\gcd(\deg x, \deg y)=1$. If $S\not \cong k\<x, y\>/(xy\pm yx)$, then every finite subgroup $G\leq \HSL(S)$ is of the form
$$G=\left\<\begin{pmatrix} \xi & 0 \\
0 & \xi^{-1}\end{pmatrix}\right\>$$
where $\xi$ is a primitive $r$-th root
of unity with $r=|G|$.
\end{lemma}

\begin{proof}
For $\s\in G$, since $\fchar
k=0$ and $\s$ is of finite order, it is easy to see that $\s=\left(\begin{smallmatrix} \lambda & 0 \\
0 & \mu \end{smallmatrix}\right)$ where $\lambda, \mu\in k^\times$ are roots of unity by
Lemma \ref{lem.Ue}. Since $\hdet \s=\lambda\mu=1$ by
Lemma \ref{lem.Ue}, 
we may write
$\s =\left(\begin{smallmatrix} \lambda & 0 \\
0 & \lambda ^{-1}\end{smallmatrix}\right)$.
If 
$r:=\max
\{|\lambda|\mid \left(\begin{smallmatrix} \lambda & 0 \\
0 & \lambda ^{-1}\end{smallmatrix}\right)\in G\}\in \NN^+$, then it is easy to see that $G=\left\<\left(\begin{smallmatrix} \xi & 0 \\
0 & \xi^{-1}\end{smallmatrix}\right)\right\>$ where $\xi$ is a primitive $r$-th root
of unity.
\end{proof}

By the above lemma, we focus on studying $S^G$ when $G=\left\<\left(\begin{smallmatrix} \xi & 0 \\
0 & \xi^{-1}\end{smallmatrix}\right)\right\>$.

\begin{lemma} \label{lem.qam}
Let $S=k\<x, y\>/(xy-yx-x^{q+1})$ where $\deg x=1, \deg y=q\in \NN^+$.  For any $1\leq r\leq q+1$ such that $\gcd(q, r)=1$, $(\cS, (r))$ is ample for $\tails S$. 
\end{lemma} 

\begin{proof} For any $1\leq m\leq q+1$, since $x^m\in S$ is a normal element,  $S/x^mS\cong S/(x^m)\cong k[x, y]/(x^m)$, so,  for any $n\geq 1$,   
$$\dim _kS/(x^mS+y^nS)=\dim _kk[x, y]/(x^m, y^n)<\infty.$$  

For any $1\leq j\leq q+1$, there exists $i\in \NN^+$ such that $1\leq ir-j\leq r\leq q+1$.  Since $\gcd(q, r)=1$, there exist $m, n\in \NN^+$ such that $nr-mq=1$, so that $\deg (y^{mj})=mqj=nrj-j$.  Define a homomorphism 
$$\begin{CD}\phi_j:S(-ir)\oplus S(-nrj) @>{(x^{ir-j}, y^{mj})\cdot }>> S(-j).\end{CD}$$
Since $1\leq ir-j\leq q+1$ and $mj\geq 1$, we have that $\operatorname {Coker} \phi _j=(S/x^{ir-j}S+y^{mj}S)(-j)$ is finite dimensional by the above argument.  Since Gorenstein parameter of $S$ is $q+1$, it follows that $\{\cS(-j)\}_{1\leq j\leq q+1}$ generates $\tails S$, 
so $(\cS, (r))$ is ample for $\tails S$ by \cite[Theorem 3.5]{Mmc} (cf. \cite[Lemma 2.7]{Mmc}). 
\end{proof} 

We expect that, for a noetherian AS-regular algebra $S$ of dimension 2, any finite subgroup $G\leq \HSL(S)$ is ample for $S$.  The following theorem shows that this is the case unless $S\cong k\<x, y\>/(xy\pm yx)$.

\begin{theorem} \label{thm.4.5} 
Let $S=k\<x, y\>/(f)$ be a noetherian AS-regular algebra of dimension 2 such
that $\gcd(\deg x, \deg y)=1$, and $G\leq \HSL(S)$ a finite subgroup.  If $G=\left\<\left(\begin{smallmatrix} \xi & 0 \\
0 & \xi^{-1}\end{smallmatrix}\right)\right\>$ where $\xi\in k^{\times}$ is a primitive $r$-th root of unity with $r=|G|$, then 
\begin{enumerate}
\item{} the map $\Phi :S*G\to \uEnd_{S^G}(S)$ defined by $\Phi(s*g)=[t\mapsto sg(t)]$ is an isomorphism of graded algebras, and 
\item{} $G$ is ample for $S$.
\end{enumerate} 
\end{theorem}

\begin{proof} (1) Since $\Phi$ preserves grading by definition, it is enough to show that $\Phi$ is an isomorphism of ungraded algebras. 

Suppose that $S=k\<x, y\>/(xy-\a yx), \a \neq 0$.  Since $G=\left\<\left(\begin{smallmatrix} \xi & 0 \\ 0 & \xi^{r-1}\end{smallmatrix}\right)\right\>$,
by regrading $S$ as $\deg x=1, \deg y=r-1$, we see that $(\cS, (r))$ becomes ample for $\tails S$ by \cite [Corollary 3.6]{Mmc}, so  
the map $\Phi :S*G\to \uEnd_{S^G}(S)$ is an isomorphism of algebras by \cite [Theorem 5.11]{U}.

Suppose that $S=k\<x, y\>/(xy-yx-x^{q+1})$ with $\deg x=1, \deg
y=q\geq 1$.  By Lemma \ref{lem.Ue}, 
$G=\left\<\left(\begin{smallmatrix} \xi  & 0 \\ 0 & \xi ^q \end{smallmatrix}\right)\right\>$ 
such that $r\mid q+1$.  Since $\gcd(q, r)=1$, $(\cS, (r))$ is ample for $\tails S$ by Lemma \ref{lem.qam}, so the map $\Phi :S*G\to \uEnd_{S^G}(S)$ is an isomorphism of graded algebras by \cite [Theorem 5.11]{U}.

(2) Since $S^G$ is a noetherian AS-Gorenstein graded isolated singularity by \cite [Corollary 5.3]{U}, and $\Phi$ is an isomorphism of graded algebras by (1), $G$ is ample for $S$ by 
Theorem \ref{thm.ampeq}.
\end{proof}

\begin{remark} If $S\not \cong k\<x, y\>/(xy\pm yx)$,
then any finite subgroup $G\leq \HSL(S)$ is of the form in the above theorem by Lemma \ref{lem.hsl}, so it is always ample.  If $S\cong k[x, y]$ with $\deg x=\deg y=1$ and $G\leq \HSL(S)=\SL(2, k)$ is a finite subgroup, then $G$ is small, 
so the map $\Phi :S*G\to \uEnd_{S^G}(S)$ is an isomorphism of graded algebras by \cite [Theorem 3.2]{IT}, and $G$ is ample by Corollary \ref{cor.ampeq}. 
For $S\cong k\<x, y\>/(xy+yx)$ with $\deg x=\deg y=1$, finite subgroups of $\HSL(S)$ were classified in \cite {CKWZ}, but we do not know if they are all ample. 
\end{remark} 

\subsection{Quiver Representations}

If $S$ is a noetherian AS-regular algebra of dimension 2, then $\nabla S\cong kQ_S$ for some
quiver $Q_S$ by \cite [Theorem 5.4]{Mmc}, so if $G\leq
\HSL(S)$ is a finite subgroup, then $(\nabla S)*G\cong kQ_S*G$ is Morita equivalent to a path algebra $kQ_{S,G}$
for some quiver $Q_{S, G}$. If $G=\left\<\left(\begin{smallmatrix} \xi & 0 \\ 0 & \xi^{-1}\end{smallmatrix}\right)\right\>$, 
then $G$ is ample, so $\cD^b(\tails S^G)\cong \cD^b(\mod kQ_{S, G})$ by Theorem \ref{thm.4.5}. The
purpose of this subsection is to compute $Q_{S, G}$ for each pair
$(S, G)$.
(We will see that, in this case, $kQ_S*G$ is isomorphic to $kQ_{S,G}$.)

We fix notations used throughout this subsection. 
Let $S=k\<x, y\>/(f)$ be a noetherian AS-regular algebra of dimension 2 and of
Gorenstein parameter $\ell=\deg x+\deg y$.
Then the quiver $Q_S$ is given as follows.
The set of vertices is $\{0,1,\dots, \ell-1 \}$.
For any vertex $i$ such that $i + \deg x < \ell$, we draw an arrow $\xymatrix@C=1pc@R=1pc{i \ar[r]^(0.3)x &i+\deg x }$.
For any vertex $i$ such that $i + \deg y < \ell$, we draw an arrow $\xymatrix@C=1pc@R=1pc{i \ar[r]^(0.3)y &i+\deg y }$.
In addition, we assume that $\gcd( \deg x, \deg y) = 1$.
Then the underlying graph of the quiver $Q_S$ is the extended Dynkin 
diagram of type $\widetilde{A}_{\ell-1}$.
Let $G\leq \HSL(S)$ be a finite subgroup. Assume that $G=\left<g \right>$ where 
$$g=\begin{pmatrix} \xi & 0 \\
0 & \xi^{-1}\end{pmatrix}$$
and $\xi$ is a primitive $r$-th root
of unity with $r=|G|$.  If $f \neq xy\pm yx$,
then $G$ is always of this form by Lemma \ref{lem.hsl}.
Put $\L = (\nabla S)*G\cong kQ_S*G$.

%
%
%
%
%
%
%

\begin{lemma} \label{lem.La}
The following hold.
\begin{enumerate}
\item Let $\{e_0, \dots, e_{\ell-1}\}$ be a complete set of primitive orthogonal idempotents of $kQ_S$,
and let $\rho_i = \frac{1}{r}\sum _{p=0}^{r-1} \xi^{ip} g^p \in kG$ for $i=0,\dots,r-1$.
Then
$$\{ e^{j}_{i} := e_i * \rho_j \mid 0\leq i\leq \ell-1, 0\leq j\leq r-1 \}$$
is a complete set of primitive orthogonal idempotents of $\L$.

\item $\L$ is basic.
\end{enumerate}
\end{lemma}

\begin{proof}
%
%
%
(1) This follows from a straightforward calculation.

(2) It is enough to show $\L /\rad \L = k \times \cdots \times k$ as algebras.
\begin{align*}
\L /\rad \L & \cong (kQ_S/\rad kQ_S)*G\\
& \cong (ke_0 \times \cdots \times ke_{\ell-1})*G\\
& \cong ke_0*G \times \cdots \times ke_{\ell-1}*G
\end{align*}
where the first isomorphism is by \cite[Section 1.2]{RR} and the last isomorphism is by
the fact that $G$ acts on $e_i$ trivially. Moreover, we have an isomorphism
\begin{align*}
ke_{i}*G &\cong e_i^0(ke_{i}*G)e_i^0 \times \cdots \times e_i^{r-1}(ke_{i}*G)e_i^{r-1};\\
e_i * g^s &\mapsto (e_i^0(e_i*g^s)e_i^0, \cdots, e_i^{r-1}(e_i*g^s)e_i^{r-1} ).
\end{align*}
Since $e_i^j(ke_{i}*G)e_i^j$ is $1$-dimensional over $k$, we obtain the desired assertion.
\end{proof}

By the above lemma and \cite[Theorem 1.3(c)]{RR}, $\L = kQ_S*G$ is basic and hereditary,
so there exists a quiver $Q_{S,G}$ such that $\L \cong k Q_{S,G}$.
We now compute the quiver $Q_{S,G}$.

\begin{theorem} \label{lem.QSG}
The quiver $Q_{S,G}$ is given as follows.
The set of vertices is $\{(i,j) \mid  0\leq i\leq \ell-1, 0\leq j\leq r-1\}$.
For any arrow $\xymatrix@C=1pc@R=1pc{ i \ar[r]^(0.2){x} &i'\;(=i+\deg x)}$ in $Q_S$ and any $j \in \ZZ/r\ZZ$, we draw an arrow $\xymatrix@C=1pc@R=1pc{ (i, j-1) \ar[r] &(i',j) }$.
For any arrow $\xymatrix@C=1pc@R=1pc{ i \ar[r]^(0.2){y} &i'\;(=i+\deg y)}$ in $Q_S$ and any $j \in \ZZ/r\ZZ$, we draw an arrow $\xymatrix@C=1pc@R=1pc{ (i, j+1) \ar[r] &(i',j) }$.
\end{theorem}

\begin{proof}
This result can be obtained by applying arguments in \cite[Section 2.3]{RR} to our setting.
\end{proof}

\begin{example}
We consider the case $\deg x=1,\deg y=1$ and $|G|=3$
(see \cite{IT}).
In this case, $\ell=2$, $r=3$.
The quiver $Q_S$ is given as follows.
\[
\xymatrix@C=2pc@R=1pc{
0 \ar@<0.5ex>[r]^{x} \ar@<-0.5ex>[r]_{y}
&1.
}
\]
Then the quiver $Q_{S,G}$ is given as follows.
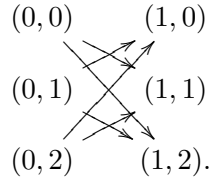
\[\xymatrix@C=1.5pc@R=0.8pc{ (0,0)&(1,0) \\
              (0,1)&(1,1)\\
               (0,2)&(1,2). 
                   \ar "1,1";"2,2"
                   \ar "1,1";"3,2"
                   \ar "2,1";"1,2"
                   \ar "2,1";"3,2"
                   \ar "3,1";"1,2"
                   \ar "3,1";"2,2" }\]
\end{example}

\begin{example}
We consider the case $\deg x=1,\deg y=3$ and $|G|=2$.
In this case, $\ell=4$, $r=2$.
The quiver $Q_S$ is given as follows.
\[
\xymatrix@C=2pc@R=1pc{
0 \ar[r]^x
&1 \ar[r]^x
&2 \ar[r]^x
&3. \ar@{<-}@/^0.5pc/@<0.3ex>[lll]^y
}
\]
Then the quiver $Q_{S,G}$ is given as follows.
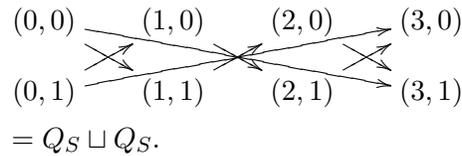
\begin{align*}
&\xymatrix@C=1.5pc@R=0.8pc{ (0,0)&(1,0)&(2,0)&(3,0)\\
              (0,1)&(1,1)&(2,1)&(3,1)
                   \ar "1,1";"2,2"
                   \ar "1,1";"2,4"
                   \ar "1,2";"2,3" 
                   \ar "1,3";"2,4"
                   \ar "2,1";"1,2"
                   \ar "2,1";"1,4"
                   \ar "2,2";"1,3"
                   \ar "2,3";"1,4" }\\
&= Q_S \sqcup Q_S.
\end{align*}
\end{example}

\begin{example}
We consider the case $\deg x=3,\deg y=5$ and $|G|=4$.
In this case, $\ell=8$, $r=4$.
The quiver $Q_S$ is given as follows.
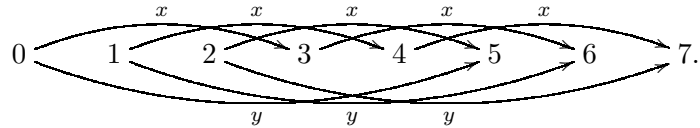
\[
\xymatrix@C=2pc@R=1pc{
0 \ar@/^1pc/[rrr]^{x} \ar@/_1.5pc/[rrrrr]_{y}
&1 \ar@/^1pc/[rrr]^{x} \ar@/_1.5pc/[rrrrr]_{y}
&2 \ar@/^1pc/[rrr]^{x} \ar@/_1.5pc/[rrrrr]_{y}
&3 \ar@/^1pc/[rrr]^{x}
&4 \ar@/^1pc/[rrr]^{x}
&5 
&6 
&7. 
}
\]
By appropriate renumbering of vertices, it is equal to 
\[
\xymatrix@C=2pc@R=1pc{
0 \ar[r]^x
&1 \ar[r]^x
&2 \ar@{<-}[r]^y
&3 \ar[r]^x
&4 \ar[r]^x
&5 \ar@{<-}[r]^y
&6 \ar[r]^x
&7. \ar@{<-}@/^0.6pc/@<0.4ex>[lllllll]^y
}
\]
Then the quiver $Q_{S,G}$ is given as follows.
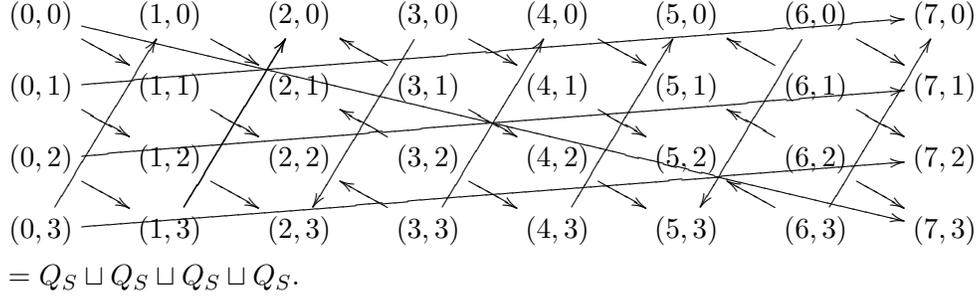
\begin{align*}
&\xymatrix@C=1.5pc@R=0.8pc{ (0,0)&(1,0)&(2,0)&(3,0)&(4,0)&(5,0)&(6,0)&(7,0)\\
             (0,1)&(1,1)&(2,1)&(3,1)&(4,1)&(5,1)&(6,1)&(7,1)\\ 
             (0,2)&(1,2)&(2,2)&(3,2)&(4,2)&(5,2)&(6,2)&(7,2)\\
             (0,3)&(1,3)&(2,3)&(3,3)&(4,3)&(5,3)&(6,3)&(7,3)
                   \ar "1,1";"2,2"
                   \ar "2,1";"3,2"
                   \ar "3,1";"4,2" 
                   \ar "4,1";"1,2"
                   \ar "1,2";"2,3"
                   \ar "2,2";"3,3"
                   \ar "3,2";"4,3"
                   \ar "4,2";"1,3" 
                   \ar "2,4";"1,3"
                   \ar "3,4";"2,3"
                   \ar "4,4";"3,3"
                   \ar "1,4";"4,3"
                   \ar "1,4";"2,5"
                   \ar "2,4";"3,5"
                   \ar "3,4";"4,5"
                   \ar "4,4";"1,5" 
                   \ar "1,5";"2,6"
                   \ar "2,5";"3,6"
                   \ar "3,5";"4,6"
                   \ar "4,5";"1,6"
                   \ar "2,7";"1,6"
                   \ar "3,7";"2,6"
                   \ar "4,7";"3,6"
                   \ar "4,2";"1,3" 
                   \ar "1,7";"4,6"
                   \ar "1,7";"2,8"
                   \ar "2,7";"3,8"
                   \ar "3,7";"4,8"
                   \ar "4,7";"1,8"
                   \ar "1,1";"4,8"
                   \ar "2,1";"1,8"
                   \ar "3,1";"2,8"
                   \ar "4,1";"3,8"}\\
&= Q_S \sqcup Q_S\sqcup Q_S\sqcup Q_S.
\end{align*}      
\end{example}

As illustrated in the above examples, 
the quiver $Q_S$ can be written as 
$$\xymatrix@C=2pc@R=1pc{
0 \ar[r]^x
&
&\cdots
& \ar[r]^x
&\ell-1 \ar@{<-}@/^0.5pc/@<0.3ex>[llll]^y
}.$$
For a quiver $Q_S$ and a positive integer $i$,
we define the $i$-covering quiver $Q_{S}^i$ as follows.
$$
 \left.
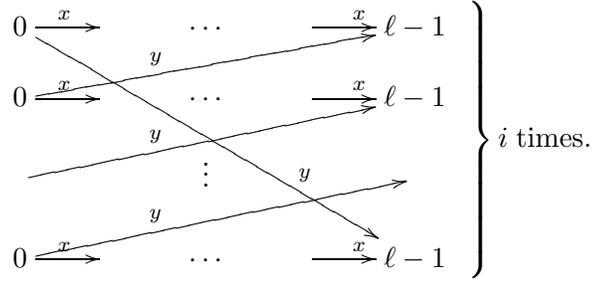
 \begin{array}{r}
  \xymatrix@C=2pc@R=1pc{
0 \ar[r]^x 
&
&\cdots
& \ar[r]^x
&\ell-1 \ar@{<-}[lllld]_(0.65)y
\\ 
0 \ar[r]^x
&
&\cdots
& \ar[r]^x
&\ell-1 \ar@{<-}[lllld]_(0.65)y
\\
&&\vdots &&\ar@{<-}[lllld]_(0.65)y
\\
0 \ar[r]^x
&
&\cdots
& \ar[r]^x
&\ell-1 \ar@{<-}[lllluuu]_(0.3)y
}
 \end{array}
 \right\} i\;\textnormal{times.}
$$
This quiver is uniquely determined by $Q_S$ and $i$.

By observation, we see that the quiver of $Q_{S,G}$ is given by
$$ \underbrace{ Q_S ^{\frac{m}{\ell}} \sqcup \cdots \sqcup Q_S^{\frac{m}{\ell}} }_{n\;\textnormal{times}}$$
where $m =\lcm(\ell,r), n=\gcd(\ell,r)$.

\begin{example} \label{ex.quiver}
We consider the case $\deg x=1,\deg y=3$ and $|G|=6$.
In this case, $\ell=4$, $r=6$.
The quiver $Q_S$ is given as follows.
\[
\xymatrix@C=2pc@R=1pc{
0 \ar[r]^x
&1 \ar[r]^x
&2 \ar[r]^x
&3. \ar@{<-}@/^0.5pc/@<0.3ex>[lll]^y
}
\]
Note that $m =\lcm(\ell,r)=12, n=\gcd(\ell,r)=2$.
Then the quiver $Q_{S,G}$ is given as follows.

\begin{align*}
&\xymatrix@C=1.5pc@R=0.8pc{ (0,0)&(1,0)&(2,0)&(3,0)\\
             (0,1)&(1,1)&(2,1)&(3,1)\\ 
             (0,2)&(1,2)&(2,2)&(3,2)\\
             (0,3)&(1,3)&(2,3)&(3,3)\\
             (0,4)&(1,4)&(2,4)&(3,4)\\
             (0,5)&(1,5)&(2,5)&(3,5)
                   \ar "1,1";"2,2"
                   \ar "2,1";"3,2"
                   \ar "3,1";"4,2" 
                   \ar "4,1";"5,2"
                   \ar "5,1";"6,2"
                   \ar "1,2";"2,3"
                   \ar "2,2";"3,3"
                   \ar "3,2";"4,3" 
                   \ar "4,2";"5,3"
                   \ar "5,2";"6,3"
                   \ar "1,3";"2,4"
                   \ar "2,3";"3,4"
                   \ar "3,3";"4,4"
                   \ar "4,3";"5,4"
                   \ar "5,3";"6,4"
                   \ar "2,1";"1,4" 
                   \ar "3,1";"2,4"
                   \ar "4,1";"3,4"
                   \ar "5,1";"4,4"
                   \ar "6,1";"5,4"
                   \ar "6,1";"1,2"
                   \ar "6,2";"1,3"
                   \ar "6,3";"1,4"
                   \ar "1,1";"6,4"}
\\
= &\xymatrix@C=1pc@R=1pc{
\bullet \ar[r]
&\bullet \ar[r]
&\bullet \ar[r]
&\bullet \ar@{<-}[r]
&\bullet \ar[r]
&\bullet \ar[r]
&\bullet \ar[r]
&\bullet \ar@{<-}[r]
&\bullet \ar[r]
&\bullet \ar[r]
&\bullet \ar[r]
&\bullet \ar@{<-}@/^0.5pc/@<0.5ex>[lllllllllll]
}\\
&\qquad \sqcup \\
&\xymatrix@C=1pc@R=1pc{
\bullet \ar[r]
&\bullet \ar[r]
&\bullet \ar[r]
&\bullet \ar@{<-}[r]
&\bullet \ar[r]
&\bullet \ar[r]
&\bullet \ar[r]
&\bullet \ar@{<-}[r]
&\bullet \ar[r]
&\bullet \ar[r]
&\bullet \ar[r]
&\bullet \ar@{<-}@/^0.5pc/@<0.5ex>[lllllllllll]
}
\\
= & Q_S^3 \sqcup Q_S^3.
\end{align*}
\end{example}

For positive integers $i, j$, we define the quiver $Q_{(i, j)}$ by
$$
\xymatrix@C=2pc@R=0.3pc{
&\bullet \ar[r]^{\a_2} &\bullet \ar[r]^{\a_3} &\cdots \ar[r]^{\a_{i-1}} &\bullet \ar[rd]^{\a_{i}}
\\
\bullet \ar[ru]^{\a_1} \ar[rd]_{\b_1} & & & & & \bullet.
\\ 
&\bullet \ar[r]_{\b_2} &\bullet \ar[r]_{\b_3} &\cdots \ar[r]_{\b_{j-1}} &\bullet \ar[ru]_{\b_{j}}
}
$$
A BGP-reflection of a quiver is a new quiver obtained by inverting all
arrows incident to a vertex which is a sink or a source.
Let $Q$ and $Q'$ be two quivers without oriented cycles.
Then the path algebras $kQ$ and $kQ'$ are derived equivalent if and
only if $Q'$ can be obtained from $Q$ by a sequence of BGP-reflections
(see \cite[Chapter \uppercase\expandafter{\romannumeral 1}, 5.7 ]{H}).

As a consequence, we obtain the following result, saying that $\tails S^G$ is determined by the triple $(\deg x, \deg y, |G|)$ up to derived equivalence.

\begin{theorem}
We have the following equivalences of triangulated categories
\begin{align*}
 \cD^b(\tails S^G)&\cong \cD^b(\mod kQ_{S, G}) \\
&\cong
\underbrace{ \cD^b(\mod k Q_{(\frac{m}{\ell}\deg x, \frac{m}{\ell}\deg y)}) \times \cdots \times \cD^b(\mod k Q_{(\frac{m}{\ell}\deg x, \frac{m}{\ell}\deg y)} )}
_{n \;\textnormal{times}}. 
\end{align*}
\end{theorem}

\begin{proof}
Since $Q_S$ has $\deg y$ arrows labeled as ``$x$'' and $\deg x$ arrows labeled as ``$y$'',
the $\frac{m}{\ell}$-covering quiver $Q_{S}^{\frac{m}{\ell}}$ has $\frac{m}{\ell} \deg y$ arrows labeled as ``$x$'' and $\frac{m}{\ell} \deg x$ arrows labeled as ``$y$''.
So by a sequence of BGP-reflection from $Q_{S}^{\frac{m}{\ell}}$, we have the quiver $Q_{(\frac{m}{\ell}\deg x, \frac{m}{\ell}\deg y)}$.
Hence it follows that
\begin{align*}
\cD^b(\mod kQ_{S, G})
&\cong \underbrace{ \cD^b(\mod kQ_S ^{\frac{m}{\ell}}) \times \cdots \times \cD^b(\mod kQ_S^{\frac{m}{\ell}})}
_{n\;\textnormal{times}}\\
&\cong \underbrace{ \cD^b(\mod k Q_{(\frac{m}{\ell}\deg x, \frac{m}{\ell}\deg y)}) \times \cdots \times \cD^b(\mod k Q_{(\frac{m}{\ell}\deg x, \frac{m}{\ell}\deg y)})}
_{n\;\textnormal{times}}.
\end{align*}
\end{proof}

\begin{remark} Since $S^G$ is a noetherian AS-Gorenstein algebra, $\tails S^G$ is a noetherian $\Ext$-finite abelian category having no non-zero projective objects and having an object of infinite length.  Moreover, since $\gldim S*G=2$, we see that $\tails S^G\cong \tails S*G$ is hereditary and having a tilting object by Theorem \ref{thm.tilt}, so the structure of $\tails S^G$ is given in \cite[Theorem 1]{L}, \cite [Theorem 6.2]{RV}.  In particular, the connected component of $\tails S^G$ is equivalent to the category of coherent sheaves on a weighted projective line.  In fact, the path algebra $kQ_{(i, j)}$ is a canonical algebra, so 
the above theorem shows that $\tails S^G$ is a product of a weighted projective line 
up to derived equivalences.  
\end{remark}

\begin{example}[The case $\deg x=1,\deg y=3$ and $|G|=6$]
This is a continuation of Example \ref{ex.quiver}.
Since
$$Q_S^3 =
\xymatrix@C=1pc@R=1pc{
\bullet \ar[r]
&\bullet \ar[r]
&\bullet \ar[r]
&\bullet \ar@{<-}[r]
&\bullet \ar[r]
&\bullet \ar[r]
&\bullet \ar[r]
&\bullet \ar@{<-}[r]
&\bullet \ar[r]
&\bullet \ar[r]
&\bullet \ar[r]
&\bullet \ar@{<-}@/^0.5pc/@<0.5ex>[lllllllllll]
},$$
we have the quiver
$$
Q_{(3,9)}=
\xymatrix@C=1pc@R=0.3pc{
&&&\bullet \ar[rrr] &&&\bullet \ar[rrrd]
\\
\bullet \ar[rrru] \ar[rd] &&&&&&&&& \bullet
\\ 
&\bullet \ar[r] &\bullet \ar[r] &\bullet\ar[r] &\bullet\ar[r] &\bullet\ar[r] &\bullet\ar[r] &\bullet\ar[r] &\bullet  \ar[ru]
}$$
by a sequence of BGP-reflections from $Q_S^3$. Hence we obtain
\begin{align*}
\cD^b(\tails S^G) \cong \cD^b(\mod k Q_{S,G} ) \cong \cD^b(\mod kQ_{(3,9)}) \times \cD^b(\mod kQ_{(3,9)}).
\end{align*}
\end{example}

\end{document}